\PassOptionsToPackage{x11names,table}{xcolor}
\documentclass[1p,number]{elsarticle}

\usepackage{array}
\usepackage{amsmath}
\usepackage{tabularx}
\usepackage{txfonts}

\usepackage[T1]{fontenc}
\usepackage[cp1250]{inputenc}

\usepackage{tikz}
\usepackage{pgfplots}
\usetikzlibrary{decorations.shapes,decorations.pathmorphing,decorations.pathreplacing}  

\usepackage{xspace}
\usepackage[T1]{fontenc}
\usepackage{enumerate}

\colorlet{cmar}{DeepSkyBlue2}
\colorlet{caga}{Goldenrod3}

\usepackage{changes}

\usepackage[colorinlistoftodos, textwidth=35mm, shadow]{todonotes}

\definechangesauthor[name={Marek}, color=cmar]{Marek}
\definechangesauthor[name={Aga}, color=caga]{Aga}

\newcommand{\R}{\varmathbb{R}}

\DeclareMathOperator{\Int}{Int}

\newtheorem{thm}{Theorem}
\newtheorem{lem}[thm]{Lemma}
\newtheorem{prop}[thm]{Proposition}

\newtheorem{rem}[thm]{Remark}

\newproof{proof}{Proof}
\newproof{prthm}{Proof of Theorem~\ref{thm:stabn}}

\numberwithin{equation}{section}

\usepackage{fancyhdr}
\newcommand{\na}[1]{{\color{red}#1}}
\newcommand{\nb}[1]{{\color{blue}#1}}
\newcommand{\nc}[1]{{\color{DarkOliveGreen3}#1}}
\newcommand{\nd}[1]{{\color{Pink2}#1}}
\newcommand{\nf}[1]{{\color{Cyan3}#1}}
\newcommand{\nh}[1]{{\color{Orange3}#1}}

\begin{document}
\begin{frontmatter}

\title{Justification of quasi-stationary approximation in models of gene expression of a~self-regulating protein}

\author[pg]{Agnieszka Bart\l{}omiejczyk\corref{cor1}}
\ead{agnbartl@pg.edu.pl}

\author[uw]{Marek Bodnar}
\ead{mbodnar@mimuw.edu.pl}

\cortext[cor1]{Corresponding author}

\address[pg]{Faculty of Applied Physics and Mathematics, Gda{\'{n}}sk University of Technology,
        Gabriela Narutowicza 11/12, 80-233 Gda{\'{n}}sk, Poland.}

\address[uw]{Institute of Applied Mathematics and Mechanics,
        University of Warsaw, Banacha 2, 02-097 Warsaw, Poland.}

\begin{abstract}
We analyse a model of Hes1 gene transcription and protein synthesis with a negative feedback loop. The effect of multiple binding sites in the Hes1 promoter as well as the dimer formation process are taken into account. 
We consider three, possibly different, time scales connected with: (i) the process of binding to/dissolving from a binding site, (ii) formation and dissociation of dimers, (iii) production and degradation of Hes1 protein and its mRNA. Assuming that the first two processes are much faster than the third one, using the Tikhonov theorem, we reduce in two steps the full model to the classical Hes1 model. In the intermediate step two different models are derived depending on the relation between the time scales of processes (i) and (ii). The asymptotic behaviour of the solutions of systems are studied. We investigate the stability of the positive steady state and perform some numerical experiments showing differences in dynamics of the considered models.
\end{abstract}


\end{frontmatter}
\section{Introduction}
Regulation of gene expression in eukaryotic cells is one of the most important processes during the life of the cell and the whole organism. How cells work depends on the signals that reach them. The cell's response is based on changing the expression of genes, and thus on the change in the amount of protein produced. Both silencing and overexpression as defects in gene regulation cause unfavorable changes. Thus, understanding the structure and mechanism of gene expression regulation is necessary to understand the functioning of many biological and chemical processes related mainly to genetic regulation. It is also necessary for understanding the emergence of diseases, and thus effective fight against them. Many cancers arise due to overexpression of major regulatory genes, i.e. genes encoding a regulatory protein. 

One such regulatory protein is Hes1 (hairy and enhancer of split 1), which belongs to the helix-loop-helix (bHLH) family of transcription proteins, i.e. DNA-binding proteins in the promoter region or in another region where regulation of transcription processes occurs. Hes1 protein deficiency in mice leads to premature cell differentiation, resulting in defects in brain tissue. In turn, the overexpression of Hes1 has been observed in many cancers, including lung cancer, ovarian cancer and colon cancer as well as germ cell tumors~\cite{Li18JC}. Hes1 also induces the activation of PARP1 in acute lymphoblastic leukemia, and patient samples during the leukemic crisis showed an elevated level of Hes1. This suggests that Hes1 protein may induce tumor cell growth. In~\cite{Liu15CBT} one more possibility of unfavorable Hes1 activity was described. The authors report that Hes1 may promote tumor metastases, including metastases to the tumor bone. This is related to the effect of Hes1 on the proliferation of cancer cells and migration abilities. Moreover, in~\cite{Li18JC} the relationship between Hes1 and breast cancer was examined in order to identify potential causes of increased invasion and metastasis of breast cancer. The authors examined Hes1 expression using Western blot analyses of freshly isolated breast cancer tissues and observed that patients with low levels of Hes1 expression have an increased survival compared to patients with high levels of Hes1 expression.

It is worth pointing out that many signaling pathways are involved in the regulation of Hes1 gene expression. It is important that Hes1 lies at the crossroads of many signaling pathways. For example Hes1 is regulated by Notch signaling pathway, which is mainly involved in the regulation of hematopoietic cell function and in tumor vasculature. Targeting in Hes1 can therefore cause fewer side effects as many other target genes of the Notch pathway will remain intact, \cite{shimojo2011,Li18JC}.

Hes1 protein as a~transcriptional repressor, inhibits its own transcription by directly binding to its 
own promoter, which blocks transcription of Hes1 mRNA~(see for instance \cite{hirata02science, sasai, takebayashi}). 
When the transcription of Hes1 mRNA~is repressed by this negative feedback, Hes1 protein soon disappears because it 
is rapidly degraded by the ubiquitin-proteasome pathway. Disappearance of Hes1 protein allows then the next 
round of transcription. In this way, Hes1 protein autonomously starts oscillatory expression induced
by a~negative feedback loop, see \cite{hirata02science}. Another example of this is the mechanism of p53 and Mdm2 proteins in which oscillations resulting from stress were observed, \cite{nasz_p53}.

The classical mathematical model which describes the gene expression of Hes 1 protein was proposed by Monk~\cite{monk03currbiol}.
This model includes four basic biochemical processes, i.e. transcription (synthesis), translation (production), protein degradation and its mRNA:
\begin{equation*}
   \begin{aligned}
      \fbox{\footnotesize{change of mRNA concentration}} &=\fbox{\footnotesize{transcription rate}} -\fbox{\footnotesize{mRNA degradation rate}}\\
      \fbox{\footnotesize{change of Hes1 protein concentration}} &=\fbox{\footnotesize{translation rate}} -\fbox{\footnotesize{Hes1 protein degradation rate}}.
   \end{aligned}
\end{equation*}
Moreover, in the model proposed by Monk~\cite{monk03currbiol} it was assumed that the intensity of mRNA production is a decreasing function 
of the concentration of the protein and the transcription time was taken into account. In~\cite{jensen03febslett} this suppression function was assumed to be a~Hill function 
with the Hill coefficient greater than two (due to the assumption that dimer binding to DNA~is co-operative one).
There are several approaches in the literature for modeling the gene expression of the hes1 protein. Hirata in~\cite{hirata02science} postulated the existence of a third non-linear component in the Hes1 model that causes oscillations. However, most models describing small autoregulation networks, such as the Hes1 model, are based on delay differential equation (DDE) and the oscillatory behaviour is caused by the delay in transcription and/or translation processes due to the Hopf bifurcation. In Bernard~et~al.~\cite{bernard06philtrans} a version of the Monk Hes1 model with delay in transcription process was considered, while in~\cite{mbab12nonrwa} we have studied the model with delay in both transcription and translation processes. In particular, we showed that the crucial factor for the appearance of oscillations is a sum of time delays in transcription and translation. Moreover, the direction of appearing Hopf bifurcation was calculated. In addition to the models based on delay differential equations, there exist models of the Hes1 regulatory pathway focusing on the spatial aspect, i.e. transport between the nucleus and the cytoplasm \cite{Sturrock13,Sturrock11jtb,Sturrock12bmb,Szymanska14jtb}. Some mathematical analysis of this model can be found in~\cite{Chaplain15,Lachowicz16dcds}. 

In this paper we consider a modification of a classical Hes1 gene expression model. Our main goal is to justify mathematically the form of the classical Hes1 model and to show that the stability of the system depends on the number of binding sites that is places at the Hes1 promoter, at which the complexes of Hes1 protein bind blocking the protein transcription. 

It is worth pointing out that before a protein binds to its own DNA, it creates a dimer that inhibits the transcription of its mRNA. For this reason, in the system we propose, we take into account the reaction describing the binding of the dimer to the binding site in the regulatory region of the gene. Since DNA promoter usually has multiple binding sites, \cite{takebayashi}, their number has a strong effect on the dynamic behaviour of the Hes1 protein system. The case of three binding sites was studied earlier by Zeiser~\cite{zeiser06tbmm}. The full model that describes dimer formation as well as many binding sites in the protein promoter leads to the system of $n+3$ ordinary differential equations (ODEs), where $n$ is the number of the binding sites. Due to the complexity of this process, we will simplify it using quasi-stationary approximation. 
Namely, we assume that the process of dimers binding to and dissociating from DNA promoter is much faster than other considered processes, and thus we consider that the probability that DNA is in active state can be described as a function of proteins' dimers.
Mathematically, this means that we can write a small parameter (or parameters) on the left-hand side of equations, and then use the Tikhonov theorem, see~\cite{tikhonov}. The Tikhonov method reduces variables of the complex systems using the assumption that some parameters are small. We emphasize once again that there are two processes here: dimer formation and binding to the promoter, and that both may have different time scales. In this way, simplification of the entire system to the classical one, we can follow two (or even three) different ways, depending on the time scales of the processes mentioned earlier, see Fig.~\ref{fig.diagram}.

\begin{figure}[h]
   \centerline{\includegraphics{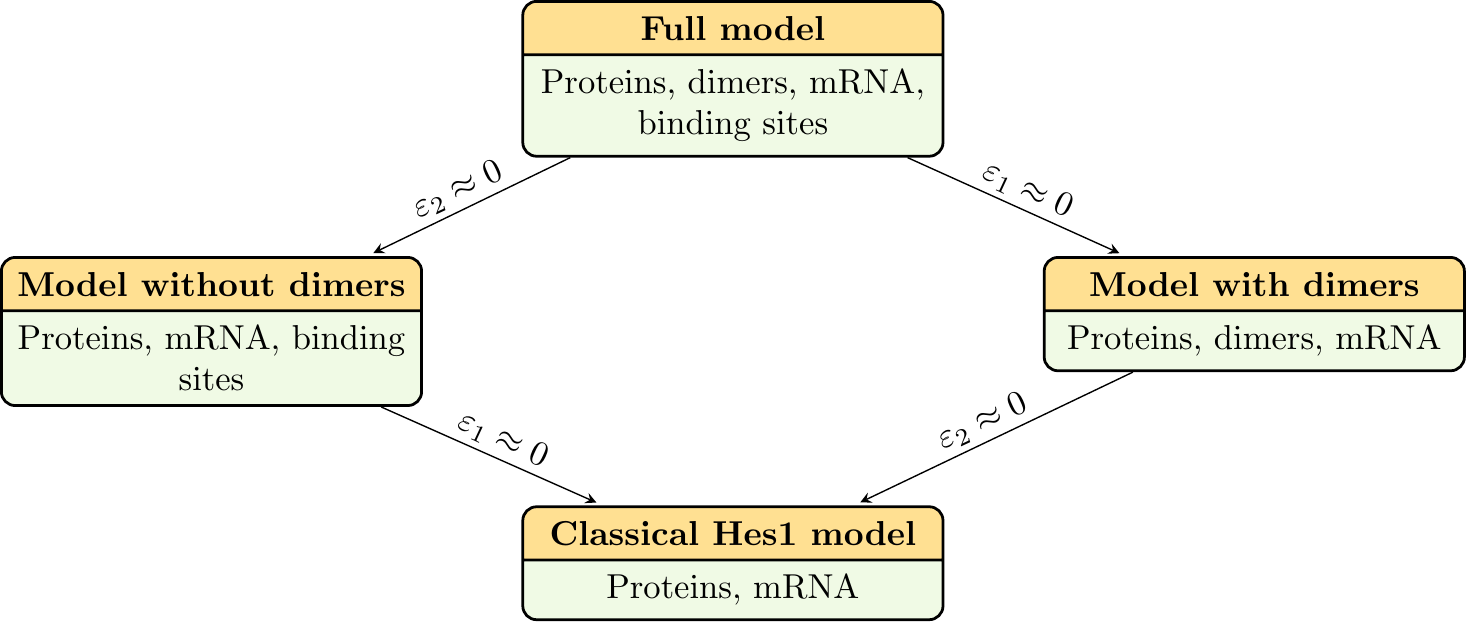}}
   \caption{A~diagram of models considered in this paper and connection between them.\label{fig.diagram}}
\end{figure}

\medskip

The paper is organized as follows. In Section~\ref{sec:fulsys} we derive a general gene expression model, we formulate its basic mathematical properties and we prove the existence of a unique positive steady state. Next, using quasi-stationary approximation we formally obtain three simplified models under different assumptions on time scales. In Section~\ref{sec:justT} using Tikhonov theorem we 
rigorously justify the form of reduced models formally derived in Section~\ref{sec:fulsys}. In Section~\ref{sec:stab} we 
study stability of the positive steady state of the reduced systems. The paper is concluded by numerical simulations and discussions presented in Section~\ref{sec:concl}.

\section{Gene expression models}\label{sec:fulsys}
\subsection{Model derivation}
Hes1 has at least three binding sites (see for instance~\cite{takebayashi}). If the Hes1 dimer binds to one of these 
sites it blocks transcription. Let us derive equations that would describe probabilities that 
a given number of sites are occupied. To this end, we assume that the concentration of Hes1 dimers 
is given by $y_2$. We construct equations for the change of the concentration of Hes1 dimers later. We show that 
it is enough to consider probabilities that $j$ sites are occupied and the particular configuration 
of free-occupied sites is not important as long as we assume that probability that Hes1 dimer bounds 
to a~free site does not depend on a particular site and is the same for all $n-j$ free sites 
(but the probability may depend on the number of free/occupied sites). 

To finish argumentation we need some notation to be 
introduced. Let $\Sigma$ denote the set of all permutations of the $n$-element set. Denote by $e_j^n$ the vector 
$(1,\ldots,1,0,\ldots,0),$ where $1$ is on the first $j$th positions and $0$ on the last $(n-j)$th positions. 
Let 
\[
e_{j,k}^n(k)=1-e_j^n(k), \quad 
e_{j,k}^n(\ell)=e_j^n(\ell), \;\; \ell \in \{1,2,\dotsc,n\}\setminus\{k\}.
\]
We see that $e_{j,k}^n$ denote the vector that differs from $e_j^n$ only on the $k$th positon.
For $\sigma\in\Sigma$ let 
$x_{\sigma(e_j^k)}$ denote the probability that the configuration of occupied and free sites is 
given by a~vector $\sigma(e_j^k)$, where $1$ on the $\ell$th coordinate of the vector $\sigma(e_j^k)$ means that 
the $\ell$th site is occupied and $0$ means that it is free. Let us 
also denote by $k_j y_2$ the~probability that a~Hes1 dimer bounds to one of free sites given that 
$j$ is occupied and let $\gamma_j$ be an intensity of dissolving of a~Hes1 dimer assuming that there are 
$j$ occupied sites. The change of the probability $x_{\sigma(e_j^k)}$ (we assume that $1\le j \le n-1$ for simplicity),
is due to one of the following actions  (see Fig.~\ref{fig.reakcja.og}): 
\begin{itemize}
\item the Hes1 dimer may bound to one of $n-j$ free sites, to each with probability $k_j y_2/(n-j)$;
\item one of $j$ bounded dimers may dissolve, each with intensity $\gamma_j$.
\end{itemize}   

\begin{figure}[!ht]
   \centerline{\includegraphics[]{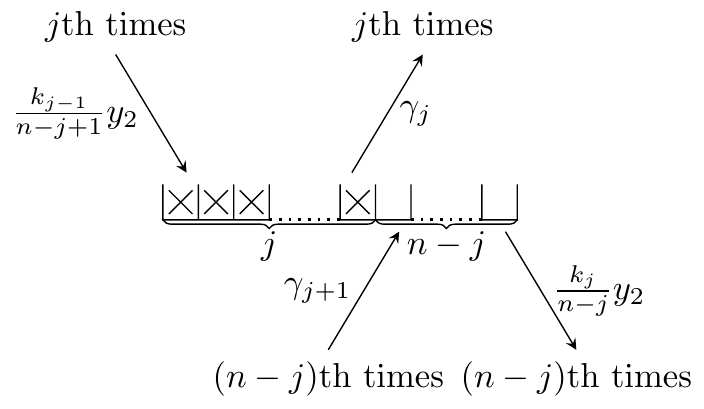}}
   \caption{The scheme of the process of binding dimers to the promoter and dissociating them. The crosses indicate occupied sites.\label{fig.reakcja.og}}
\end{figure}

Using this assumptions and the mass action law we write 
\begin{equation}\label{ogolny}
	x'_{\sigma(e_j^n)} = \sum_{\ell=1}^j \frac{k_{j-1}}{n-j+1} y_2x_{\sigma(e_{j,\ell}^n)} 
				+ \sum_{\ell=j+1}^n \gamma_{j+1} x_{\sigma(e_{j,\ell}^n)} 
				-\bigl(k_jy_2+j\gamma_j\bigr)x_{\sigma(e_j^n)},
\end{equation}
for $1\le j \le n-1$. For the cases $j=0$ and $j=n$, we have 
\begin{equation}\label{ogolny:brzeg}
	\begin{split}
		x'_{e_0^n} &= \sum_{\ell=1}^n \gamma_{1} x_{e_{0,\ell}^n} -k_0y_2x_{e_0^n}, \\
		x'_{e_n^n} &= \sum_{\ell=1}^n  k_{n-1} y_2 x_{e_{n,\ell}^n} 
				-n\gamma_n x_{e_n^n}.
	\end{split}
\end{equation}
Now, let us denote 
\[
	x_j = \sum_{\sigma\in\Sigma} x_{\sigma(e_j^n)}. 
\]
We note that $\sigma(e_{j,k}^n)$ is one of $\sigma(e_{j-1,k}^n)$ if $1\le k\le j$ 
and one of $\sigma(e_{j+1,k}^n)$ if $j+1\le k\le n$. This, together with the fact that 
Eqs.~\eqref{ogolny} and \eqref{ogolny:brzeg} are linear with respect to $\sigma(e_{j,k}^n),$
leads to the following equations 
\begin{equation}\label{ogolny:red}
	\begin{split}
		x_0' & = \gamma_1 x_1 - k_0y_2 x_0,\\
		x_j' & = k_{j-1}x_{j-1}y_2 + (j+1)\gamma_{j+1}x_{j+1} - \bigl(k_jy_2+j\gamma_j\bigr)x_j, \quad 1\le j \le n-1,\\
		x_n' & = k_{n-1}y_2 x_{n-1} - n\gamma_n x_n
	\end{split}
\end{equation}
and
\begin{equation}\label{trojka:rest_ogolny}
	\begin{split}
		y_1'&=2\gamma_yy_2-2k_y\,y_1^2+r_y z-\delta_y\,y_1,\\
		y_2'&=-\sum_{j=0}^{n-1}k_jx_jy_2+\sum_{j=1}^{n}	j\gamma_j\,x_j-\gamma_yy_2+k_y\,y_1^2,\\
		z'&=r_z x_0-\delta_{z}\,z,
	\end{split}
\end{equation}
where $\gamma_y$ denotes an intensity of dissolving a~Hes1 dimer, $k_y$ denotes an intensity of a~formation
of Hes1 dimers, $r_y$ and $r_z$ are production rates of Hes1 and its mRNA, respectively, while 
$\delta_y$ and $\delta_z$ are degradation rates of Hes1 and its mRNA, respectively. 

\subsection{Non-dimensionalisation and basic mathematical properties}
Before analysing model~\eqref{ogolny:red}--\eqref{trojka:rest_ogolny} we express it in non-dimensional terms, thereby reducing the number of parameters.

The right-hand side of system~\eqref{ogolny:red}--\eqref{trojka:rest_ogolny} 
is a~polynomial, thus existence of a~unique solution to this system is immediate. The fact that $x_0+ x_1+\cdots+x_n=1$ 
implies that system~\eqref{ogolny:red}--\eqref{trojka:rest_ogolny} can be reduced to a~system of $n+3$ equations. However, 
we find it convenient to write the system in the present perturbed form.
We show that this problem can be reduce to a~lower-dimensional problem.
We choose the scalling in such a way, that the positive steady state (which is a unique steady state of system~\eqref{ogolny:red}--\eqref{trojka:rest_ogolny}, as we prove later) has a very simple coordinates. In order to do that we proceed in the following manner.
Let $q$ be the positive solution to the following equation
\[
\frac{\delta_y\delta_z}{r_yr_z}q=\frac{1}{1+\sum_{j=1}^{n}\frac{1}{j!}\frac{k_0\ldots k_{j-1}}{\gamma_1\ldots\gamma_j}\left(\frac{k_y}{\gamma_y}q^2\right)^j}.
\]
We introduce non-dimensional quantities putting
   \[
   \tilde x_j=x_j,\quad
   \tilde y_1=\frac{y_1}{q},\quad 
   \tilde y_2=\frac{\gamma_y y_2}{k_y q^2},\quad 
   \tilde z=\frac{r_y}{\delta_y\,q}z,\quad 
   \tau= k_yq^2\,t
   \]
and
   \begin{align*}
   k&=\frac{2}{q}, &
   \delta_1&=\frac{\delta_y}{k_yq^2},& 
   \tilde\gamma_j&=\gamma_j\frac{\gamma_y}{k_0k_yq^2},&
   \varepsilon_1 &=\frac{\gamma_y}{k_0}  &
   \tilde k_j&= \frac{k_j}{k_0},\\   
   r_0&= \frac{r_yr_z}{\delta_y\delta_z\, q}, &
   \delta_2&=\frac{\delta_z}{k_yq^2},& 
   \theta &=\frac{k_0}{\gamma_y}, &
    \varepsilon_2 &=\frac{k_yq^2}{\gamma_y}
   \end{align*}
for $0\leq j\leq n.$ Note that $\tilde k_0 =1.$ 
Nevertheless, we keep writing $\tilde k_0$ to get natural and (in some sense) more symmetric formulas.

\begin{rem}
We observe, that due to the definitions of $q$ and $r_0$ dimensionless parameters fulfil the following equality
 \begin{equation}\label{rown_r0}
   r_0=1+\sum_{j=1}^{n}\frac{1}{j!}\frac{\tilde k_0\ldots \tilde k_{j-1}}{\tilde \gamma_1\ldots\tilde \gamma_j}.
   \end{equation}
\end{rem}
For notational simplicity we drop the tilde on $x_j,$ $y_1,$ $y_2,$ $z,$ $k_j$ and $\gamma_j$ ($0\leq j\leq n$) and, in consequence, the non-dimensional version of the system~\eqref{ogolny:red}--\eqref{trojka:rest_ogolny} reads as follows 
\begin{equation}\label{ogolny:epsilony_red}
   \begin{split}
   \varepsilon_1 x_0'&=\gamma_1x_1- k_0 x_0 y_2,\\
   \varepsilon_1 x_j'&= k_{j-1}x_{j-1}y_2+(j+1)\gamma_{j+1}x_{j+1}-\left( k_jy_2+j\gamma_j\right)x_j,\quad 1\leq j\leq n-2,\\
   \varepsilon_1 x_{n-1}' & = k_{n-2}x_{n-2}y_2 + n\gamma_{n}\left(1-\sum_{j=0}^{n-1}x_j\right) - \bigl(k_{n-1}y_2+(n-1)\gamma_{n-1}\bigr)x_{n-1},\\
   y_1'&=k(y_2-y_1^2)+\delta_1\bigl(z-y_1\bigr),\\
   \varepsilon_2 y_2'&=\theta\left(	-\sum_{j=0}^{n-1} k_jx_jy_2+ \sum_{j=1}^{n-1}j\gamma_jx_j+n\gamma_n\biggl(1- \sum_{j=0}^{n-1}x_j\biggr)\right)-y_2+y_1^2,\\
   z'&=\delta_2\bigl(r_0 x_0- z\bigr).
   \end{split}
\end{equation}
From now on we deal with dimensionless model.

\subsubsection{Uniqueness, existence, and non-negativity of solutions and existence of a unique positive steady state}

\begin{thm} \label{invariant}
   The solutions to~\eqref{ogolny:epsilony_red} exist are nonnegative, unique and defined for all $t$.
   Every set 
   \[
   \Omega = \Bigl\{ (x_0,\ldots,x_{n-1},y_1,y_2,z)\in  \R^{n+3} : 0\le x_j,\, \sum_{j=0}^{n-1}x_j\le 1,\, 0\le y_1\le \bar y_1,\, 0\le y_2\le \bar y_2,\, 0\le z\le r_0,\, 0\leq j\leq n-1\Bigr\} 	
   \] 
   such that constants $\bar y_1$, $\bar y_2$ fulfil
   \begin{equation}\label{ogr}
   \bar y_1 = r_0+\frac{k}{\delta_1}\theta_{\gamma}, 
   \quad
   \bar y_2 = \bar y_1^2 + \theta_{\gamma},  \quad \theta_{\gamma} = \theta \sum_{j=1}^{n}j\gamma_j,
   \end{equation}
   is invariant for the evolution system~\eqref{ogolny:epsilony_red}.
\end{thm}
\begin{proof}
   Since $\sum_{i=0}^n x_i(t) = 1$ and all variables are positive we immediately have $0\le x_0(t)\le 1$. Thus, from the last equation
   of~\eqref{ogolny:epsilony_red}, we get 
   \[
   z'(t) \le \delta_2\bigl(r_0 - z\bigr)\; \Longrightarrow \; z(t) \le \max\Bigl\{z(0),r_0\Bigr\}.
   \]
   If the initial condition is from $\Omega$ then $z(t)\le r_0$. 
   For the similar reason, the equations for $y_1$ and $y_2$ can be estimated as
   \[
   y_1' \le k(y_2-y_1^2)+ \delta_1\bigl(r_0-y_1\bigr), \quad
   y_2' \le 
   \theta\left(	\sum_{j=1}^{n-1}j\gamma_jx_j+n\gamma_n\right)-y_2+y_1^2\le 
   = \theta_\gamma-y_2+y_1^2,
   \]
   where we used the definition of $\theta_\gamma$. 
   Let draw two curves
   \begin{equation}\label{pseudo_izo}
   y_2 = y_1^2+\theta_{\gamma}, \quad 
   y_2 = y_1^2+\frac{\delta_1}{k}\Bigl(y_1-r_0\Bigr),
   \end{equation}
   in the $(y_1,y_2)$ plane (see the solid red and dashed blue curve, respectively, in Fig.~\ref{fig.izo}). Note, that $y_1$ 
   decreases below the second curve while $y_2$ decreases above the first one. Note also that these two curves intersect 
   at the point 
   \[
   \bar y_1 = r_0+\frac{k}{\delta_1}\theta_{\gamma}>0.
   \]
   The line connecting the point  $(0,\bar y_2)$ with $(\bar y_1,\bar y_2)$ is above the curve
   $y_2 = y_1^2+\theta_{\gamma}$ (the solid red line in Fig.~\ref{fig.izo}) and thus, $y_2'<0$.
   The line connecting the point $(\bar y_1,0)$ with $(\bar y_1, \bar y_2)$ is below the curve
   $y_2 = y_1^2+\frac{\delta_1}{k}\Bigl(y_1-r_0\Bigr)$ 
   (the dashed blue line in Fig.~\ref{fig.izo}) and thus, $y_1'<0$. 
   This shows that $y_1(t)$, $y_2(t)$ cannot escape the region bounded by these hyperplanes, which completes 
   the proof.
      
   \begin{figure}[!htb]
      \centerline{\includegraphics[]{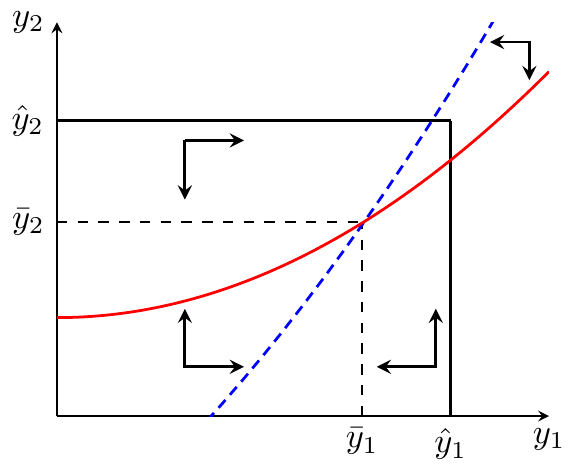}}
      \caption{The sketch of the curves given by~\eqref{pseudo_izo}. The arrows denote the direction 
         of the vector field.\label{fig.izo}
      }
   \end{figure}
	\qed
\end{proof}

   \begin{rem} 
      For any constants $\hat y_1$, $\hat y_2$ fulfils 
      $\bar y_1 < \hat y_1$, $\hat y_1^2+\theta_{\gamma}< \hat y_2 < \hat y_1^2+\frac{\delta_1}{k}\Bigl(\hat y_1-r_0\Bigr)$
      a set 
      \[
      \hat \Omega = \Bigl\{ (x_0,\ldots,x_{n-1},y_1,y_2,z)\in  \R^{n+3} : 0\le x_i,\, \sum_{i=0}^{n-1}x_i\le 1,\, 0\le y_1\le \hat{y}_1,\, 0\le y_2\le \hat y_2,\, 0\le z\le r_0,\, 0\leq i\leq n-1\Bigr\} 	
      \] 
      is invariant under the evolution of system~\eqref{ogolny:epsilony_red}. Note also that the inequality $\bar y_1 < \hat y_1$ implies that $k\theta_\gamma<\delta_1\bigl(\hat y_1-r_0\bigr)$.
   \end{rem}

\begin{prop}
There exists exactly one non-negative steady state $(\bar x_0, \bar x_1,\ldots\bar x_{n-1},1,1,1)$ of system~\eqref{ogolny:epsilony_red}, where
   \begin{equation}\label{ss:ogolny}
   \bar{x}_{0}=\frac{1}{r_0},\quad \bar{x}_{j}=\frac{1}{j!}\cdot\frac{ k_0 k_1\cdots  k_{j-1}}{\gamma_1\gamma_2\cdots \gamma_j}\cdot \frac{1}{r_0}, \; 1\leq j\leq n-1, \quad \bar{y}_1 =1, \quad \bar{y}_2 =1, \quad \bar{z}=1.
   \end{equation}
\end{prop}
\begin{proof}
    Note that the first $n$ equations are linear with respect to $x_0, x_1, \dotsc, x_{n-1}$ if $y_2$ is fixed. Thus, looking for a steady state 
    $(\bar x_0, \bar x_1, \dotsc, \bar x_{n-1}, \bar y_1, \bar y_2, \bar z)$ of system~\eqref{ogolny:epsilony_red} we easily get 
    \begin{equation}\label{rown_na_xi}
      \bar x_j = \frac{1}{j}\cdot\frac{k_{j-1}}{\gamma_{j}} \bar y_2 \bar x_{j-1}
         =\frac{1}{j!}\cdot\frac{ k_0 k_1\cdots  k_{j-1}}{\gamma_1\gamma_2\cdots \gamma_j} \bar y_2^j \bar x_0 , \quad j=1,2,\dotsc,n,
    \end{equation}
   and $\bar x_n = 1 - (\bar x_0+\bar x_1 + \dotsb + \bar x_{n-1})$. 
    
    Now, it is easy to see that the following identity 
    \begin{equation}\label{stany_stac1}
      \bar x_0  = \frac{1}{1+\sum_{j=1}^{n}\frac{1}{j!}\frac{k_0\ldots k_{j-1}}{\gamma_1\ldots\gamma_j}\bar y_2^j}
    \end{equation}
    holds.
    From the last equation of~\eqref{ogolny:epsilony_red} we deduce that 
    \[
      \bar z = r_0 \bar x_0.
    \]
    The equation for $y_2$ implies that $\bar y_2 = \bar y_1^2$ (the expression
    that is multiplied by $\theta$ is equal to zero due to~\eqref{stany_stac1}) and therefore
    \[
      \bar y_1 = \bar z.
      \]
    Combing this we have 
    \begin{equation}\label{rown_na_y1}
      \bar y_1  = \frac{r_0}{1+\sum_{j=1}^{n}\frac{1}{j!}\frac{k_0\ldots k_{j-1}}{\gamma_1\ldots\gamma_j}\bar y_1^{2j}}.
    \end{equation}
    We see at once that equation~\eqref{rown_na_y1} has a unique solution since its right-hand side is a positive decreasing function of $y_1$, its left-hand side is a linear increasing function taking value 0 at $\bar y_1=0$. Due to~\eqref{rown_r0} it is immediate that $\bar y_1=1$ solves~\eqref{rown_na_y1} and then formulas~\eqref{ss:ogolny} follow.
\qed
\end{proof}

\subsection{Formal derivation of reduced models}
The classical Hes1 gene expression model proposed by Monk~\cite{monk03currbiol} is a system of only two differential equations that describes concentrations of Hes1 mRNA and Hes1 proteins. In this section we reduce model~\eqref{ogolny:epsilony_red} to the classical Hes1 model in the two consecutive steps by means of the quasi-stationary approximation. We consider that the complex system described by~\eqref{ogolny:epsilony_red} consists of three natural time scales that are connected with: production of mRNA and protein, dimer formation, dimer binding to and dissolving from the DNA promoter. The second and the third time scales are reflected by $\varepsilon_2$ and $\varepsilon_1$, respectively. Assuming that $\varepsilon_1, \varepsilon_2\ll 1$, the reduction can be done in three different ways 
depending on whether $\varepsilon_2\ll\varepsilon_1$ or $\varepsilon_1\ll\varepsilon_2$ or $\varepsilon_1\approx\varepsilon_2$. 
The last case is more complex and will be considered elsewhere.
    
In the first case we set $\varepsilon_2\approx 0$ and $\varepsilon_1>0$ obtaining the system in which the concentration of dimers is always in a stationary level. Then we put $\varepsilon_1\approx 0$ which reduces the system to the classical one. In the second case we first set $\varepsilon_1\approx 0$ and $\varepsilon_2>0$ obtaining system in which the probability that DNA is not blocked as a function of hes1 dimers concentration. Then we set $\varepsilon_2\approx 0$ obtaining again the classical system. These two approaches are schematically illustrated in Fig.~\ref{fig.diagram}.

\subsubsection{Derivation of a simplified model under the assumption that dimer dynamics is much faster than other ones.} 
We start from the reduction of the model \eqref{ogolny:epsilony_red} assuming that dynamics of dimer formation and dissociation is much faster than other processes. In this case $\varepsilon_2$ is very small, thus we consider $\varepsilon_2\to 0$. Now, we derive a simplified model. To this end, 
we set $\varepsilon_2\approx 0$ and we calculate $y_2$ in dependence on other variables obtaining 
\begin{equation}\label{y2stacj}
y_2 = \frac{y_1^2+\theta \sum_{j=1}^{n-1}j\gamma_jx_j+\theta n\gamma_n\Bigl(1-\sum_{j=0}^{n-1}x_j\Bigr)}
{1+\theta\sum_{j=0}^{n-1}k_jx_j}=:\varphi(x,y_1),
\end{equation} 
where $x=(x_0,\dotsc,x_{n-1})$.
The expression above describes the stationary concentration of Hes1 dimers when other quantities are given.
The rest of the equations of \eqref{ogolny:epsilony_red} are exactly as before, i.e. 
\begin{equation}\label{uproszcz:n+2}
\begin{split}
\varepsilon_1 x_0'&=\gamma_1x_1- k_0 x_0 \varphi(x,y_1),\\
\varepsilon_1 x_j'&= k_{j-1}x_{j-1}\varphi(x,y_1)+(j+1)\gamma_{j+1}x_{j+1}-\left( k_j\varphi(x,y_1)+j\gamma_j\right)x_j,\quad 1\leq j\leq n-2,\\
\varepsilon_1 x_{n-1}' & = k_{n-2}x_{n-2}\varphi(x,y_1) + n\gamma_{n}\left(1-\sum_{j=0}^{n-1}x_j\right) - \bigl(k_{n-1}\varphi(x,y_1)+(n-1)\gamma_{n-1}\bigr)x_{n-1},\\
y_1'&=k(\varphi(x,y_1)-y_1^2)+\delta_1\bigl(z-y_1\bigr),\\
z'&=\delta_2\bigl(r_0x_0- z\bigr),
\end{split}
\end{equation}
where $\varphi(x,y_1)$ is given by~\eqref{y2stacj}.

\subsubsection{Derivation of a simplified model under the assumption that dynamics of free and occupied sites is much faster than other ones.} \label{sec:1to2R}
In order to simplify the system~\eqref{ogolny:epsilony_red}, the first $n$-equations can be reduced to algebraic equations. 
Setting $\varepsilon_1\approx 0$ we obtain the system
\begin{equation}\label{ogolny:algebraic_bez1}
   \begin{split}
   &\gamma_1x_1- k_0 y_2x_0=0,\\
   &k_{j-1}y_2x_{j-1}+(j+1)\gamma_{j+1}x_{j+1}-\left( k_jy_2+j\gamma_j\right)x_j=0,\quad 1\leq j\leq n-2,\\
   &k_{n-2}x_{n-2}y_2 + n\gamma_{n}\left(1-\sum_{j=0}^{n-1}x_j\right) - \bigl(k_{n-1}y_2+(n-1)\gamma_{n-1}\bigr)x_{n-1}=0.
   \end{split}
\end{equation}   
Treating $y_2$ as a~parameter we solve~\eqref{ogolny:algebraic_bez1} obtaining
\begin{equation}\label{ogolny:wzor_y2}
\begin{split}
x_0 & = \psi(y_2),\\
x_j & =\frac{1}{j!}\frac{k_0\ldots k_{j-1}}{\gamma_1\ldots\gamma_j}\,y_2^j\,\psi(y_2), \quad 1\le j \le n-1,
\end{split}
\end{equation}
where
\begin{equation}\label{bart:wzor_psi_ogolny}
\psi(y_2)=\frac{1}{1+\sum_{j=1}^{n}\frac{1}{j!}\frac{k_0\ldots k_{j-1}}{\gamma_1\ldots\gamma_j}y_2^j}.
\end{equation}
Moreover, we see at once that 
\[
-\sum_{j=0}^{n-1}k_jx_jy_2+\sum_{j=1}^{n-1}j\gamma_jx_j + n\gamma_n\left(1- \sum_{j=0}^{n-1}x_j\right)=0,
\]
which is clear from~\eqref{ogolny:algebraic_bez1}.   
Finally, we get the following reduced system
\begin{equation}\label{uposzcz:n}
\begin{split}
y_1'&=k(y_2-y_1^2)+\delta_1\bigl(z-y_1\bigr),\\
\varepsilon_2 y_2'&=y_1^2-y_2,\\	
z'&= \delta_2\Bigl(r_0 \psi(y_2)-z\Bigr), 
\end{split}
\end{equation}
where the function $\psi$ is given by~\eqref{bart:wzor_psi_ogolny}.

\subsubsection{Reduction of system~\eqref{uproszcz:n+2} and~\eqref{uposzcz:n} into the classical Hes1 gene expression model}
Here, we derive the classical Hes1 gene expression model starting from \eqref{uproszcz:n+2} and \eqref{uposzcz:n}. We first consider~\eqref{uposzcz:n}.
We observe that the left-hand side of the second equation of \eqref{uposzcz:n} is multiplied by a~small parameter $\varepsilon_2,$
since the creation and dissociation of Hes1 dimers are much faster than translation and transcription processes.
We derive now, the dependence of $y_2$ on $y_1$. So, putting $\varepsilon_2\approx 0$ and using the second equation of~\eqref{uposzcz:n}, we get
\[
y_2 = y_1^2,
\]
which leads to the reduced system
\begin{equation}\label{zreduk}
\begin{split}
y_1'&=\delta_1(z-y_1),\\
z'&=\delta_2\Bigl(r_0\psi(y_1^2)-z\Bigr).
\end{split}
\end{equation}
On the other hand, putting $\varepsilon_1\approx 0$ in~\eqref{uproszcz:n+2} and proceeding as in Section~\ref{sec:1to2R}
we also arrive to system~\eqref{zreduk}.
Model~\eqref{zreduk} is well known in the literature and we only cite the stability result 
that can be found, for example, in~\cite{mbab12nonrwa}. We cite it here for clarity.
\begin{prop}
   If the function $\psi$ is strictly decreasing, 
   the positive steady state of the system~\eqref{zreduk} is locally asymptotically stable. 
\end{prop}
Of course, the function $\psi$ given by~\eqref{bart:wzor_psi_ogolny} is a~decreasing function. 
\begin{rem}
   In the classical Hes1 model (see~\cite{jensen03febslett}), the term $\psi(y_1^2)$ is replaced by the Hill function
   \begin{equation}\label{funkcjaHilla}
   \psi_h(y_1^2) \approx \frac{a^h}{a^h+y_1^h},
   \end{equation}
   where $h$ is the Hill coefficient.
\end{rem}

\section{Justification of quasi-stationary approximation using the Tikhonov theorem }\label{sec:justT}

\subsection{The Tikhonov theorem}
Assume that in a system of ordinary differential equations there exists a subsystem, which dynamics is much faster than others equations.
If such ``fast'' subsystem can be distinguished, then the naive thinking is that this fast subsystem is close to its stationary state so we can eliminate from one to several variables from the full system and replace it by some algebraic equations. However, it is clear that such approach need not always to be true --- imagine for example that the steady state of the fast subsystem is unstable. In this Section, we remind a mathematical theory that justify the approximation described above. It is based on the Tikhonov theorem of 1952, (see~\cite{BanasiakLachowicz, tikhonov}), which states that as small parameter $\varepsilon>0$ converges to zero, the solution of the full system approaches the solution of the degenerate (slow) system.
More precisely, the Tikhonov theorem implies that the solutions of full system can be approximated by the solutions of the reduced system. In order to make the paper clearer we cite here the Tikhonov theorem together with needed assumptions following~\cite{BanasiakLachowicz}.
We consider the following system of ODEs with one small parameter	
   \begin{equation}\label{Tikhonov}
		\begin{split}
			u'(t) & =  F(u,v),\quad  u(0)={u}_0,\\
			\varepsilon v'(t) & =  G(u,v),\quad  v(0)=v_0.
		\end{split}
\end{equation}
System~\eqref{Tikhonov} consist of two subsystems: equations for $u$ (which is usually called a slow system, and equations for $v$, which is called fast subsystem. 
To formulate the Tikhonov theorem we need the following assumptions:
\begin{enumerate}[({A}1)]
   \item $F\colon\Omega\to\R^n$ and $G\colon\Omega\to\R^m$ are continuous and satisfy the Lipschitz condition in $\Omega$, where $\Omega=\overline{U}\times V$ is a subset of $\R^{n+m}$, where $\overline{U}$ is a compact set in $R^n$ and $V$ is a bounded open set in $R^m$,
   \item for any $u\in\overline{U}$ there exists an isolated solution $v=\phi(u)\in V$ of the algebraic equation $G(u,v)=0$ and $\phi$ is continuous,
   \item for any $u\in\overline{U}$ treated as a parameter the solution of the initial layer equation	$v'(t)=G(u,v)$
	is asymptotically stable (uniformly with respect to the $u$),
   \item the function $u\mapsto F(u,\phi(u))$ satisfies the Lipschitz condition with respect to $u$ in $\overline{U}$
      and there exists a unique solution $\bar u(t)$ of the reduced system 
      \begin{equation}\label{Tikhonov_red}
         u'(t)=F(u,\phi(u)),\quad u(0)={u}_0
      \end{equation}
      such that $\bar u(t)\in \Int\,\overline{U}$ for all $t\in (0,T),$
   \item $v_0$ belongs to the region of attraction of the point $\phi(u_0),$ where $G(u_0,\phi(u_0))=0,$ i.e. the solution $\hat v=\hat{v}(t)$ of the initial problem
      \[
         v'(t)=G(u_0,v),\quad v(0)=v_0
      \]
      satisfies $\lim_{t\to\infty}\hat{v}(t)=\phi(u_0).$
\end{enumerate}

\begin{thm} [Tikhonov] \label{thm:Tikhonov}
Let $T>0$ be an arbitrary number.
Under the assumptions {\upshape(A1)}--{\upshape(A5)} there exists $\varepsilon_0 >0,$ such that
for any $\varepsilon\in(0,\varepsilon_0]$ there exists a unique solution $(u_\varepsilon(t),v_\varepsilon(t))$ of the full system~\eqref{Tikhonov} on $[0,T]$ and
\begin{align*}
\lim_{\varepsilon\to 0} u_\varepsilon(t)&= \bar u(t),\quad t\in [0,T],\\
\lim_{\varepsilon\to 0} v_\varepsilon(t)&= \bar v(t),\quad t\in (0,T],
\end{align*}
where $\bar{u}(t)$ is the solution of the reduced problem~\eqref{Tikhonov_red} and $\bar v(t)=\phi(\bar u(t))$.
 \end{thm}
The Tikhonov theorem gives the conditions under which the solution $(u_\varepsilon(t),v_\varepsilon(t))$ to 
system~\eqref{Tikhonov} converges to $(\bar u(t),\bar v(t))$, where 
\begin{itemize}
   \item $\bar v$ is the solution of the algebraic equation $0=G(u,v),$
   \item $\bar u$ is the solution to~\eqref{Tikhonov_red} obtained from the first equation of system~\eqref{Tikhonov} by substituting a known quasi-stationary solution $\bar v$ instead of $v$.  
\end{itemize}

Now, using the Tikhonov theorem we prove that for $\varepsilon_1$ small enough solution to 
system~\eqref{uposzcz:n} (the model with dimers) approximates solution to system~\eqref{ogolny:epsilony_red} (the full model), and solution to~\eqref{zreduk} (the classical Hes1 model) approximates solutions to~\eqref{uproszcz:n+2} (the model without dimers), see Fig.~\ref{fig.diagram}. We also show that for $\varepsilon_2$ small enough solution to~\eqref{uproszcz:n+2} (the model without dimers) approximates solution to system~\eqref{ogolny:epsilony_red} (the full model), and solution to~\eqref{zreduk} (the classical Hes1 model) approximates solutions to~\eqref{uposzcz:n} (the model with dimers), see the upper left and lower right arrows at Fig.~\ref{fig.diagram}.

To this end we need to introduce some notation. Let
\begin{equation*}
\begin{split}
\Omega_x&=\Bigl\{ (x_0,\ldots,x_{n-1})\in\R^{n} : 0\le x_j,\, \sum_{j=0}^{n-1}x_j\le 1,\, 0\le j\le n-1\Bigr\},\\
\Omega_{y_2}&=\Bigl\{ y_2\in\R: y_2\le \bar y_2\Bigr\},\\
\Omega_w&=\Bigl\{ (y_1,z)\in  \R^{2} : 0\le y_1\le \bar y_1,\, 0\le z\le r_0\Bigr\},
\end{split}
\end{equation*}
where $\bar y_1$ and $\bar y_2$ are given by formula~\eqref{ogr}.
We rewrite the full system~\eqref{ogolny:epsilony_red} in the following way
\begin{equation}
\begin{split}\label{nasz}
\varepsilon_1 x'&=f(x,y_2),\\
\varepsilon_2 y_2'&=g(x,y_2,w),\\
w'&=h(x,y_2,w),
\end{split}
\end{equation}
with initial condition
\begin{equation}\label{nasz_pocz}
x(0)=\mathring{x},\quad
y_2(0)=\mathring{y}_2,\quad
w(0)=\mathring{w},
\end{equation}
where $x=(x_0,x_1,\dotsc,x_{n-1})$, $w=(y_1,z),$ $\mathring{w}=(\mathring{y}_1,\mathring{z})$ and the functions $f,$ $g,$ $h$ are given by the following forms
\begin{equation}\label{function_f}
f(x,y_2)= \left[ \begin{array}{ccc}
f_0(x,y_2)\\
...\\
f_j(x,y_2)\\
...\\
f_{n-1}(x,y_2)
\end{array} \right]=
\left[ \begin{array}{ccc}
\gamma_1x_1- k_0 x_0 y_2\\
...\\
k_{j-1}x_{j-1}y_2+(j+1)\gamma_{j+1}x_{j+1}-\left( k_jy_2+j\gamma_j\right)x_j\\
...\\
k_{n-2}x_{n-2}y_2 + n\gamma_{n}\left(1-\sum_{j=0}^{n-1}x_j\right) - \bigl(k_{n-1}y_2+(n-1)\gamma_{n-1}\bigr)x_{n-1}
\end{array} \right],
\end{equation}
\begin{equation}\label{function_g}
g(x,y_2,w)= \theta\left(	-\sum_{j=0}^{n-1} k_jx_jy_2+ \sum_{j=1}^{n-1}j\gamma_jx_j+n\gamma_n\biggl(1- \sum_{j=0}^{n-1}x_j\biggr)\right)-y_2+y_1^2
\end{equation}
and
\begin{equation}\label{function_h}
h(x,y_2,w)= \left[ \begin{array}{ccc}
k(y_2-y_1^2)+\delta_1\bigl(z-y_1\bigr)\\
\delta_2\bigl(r_0 x_0- z\bigr)
\end{array} \right].
\end{equation}

\begin{rem}
   We observe that the functions $f:\Omega_x\times\Omega_{y_2}\to \R$ and $g,h:\Omega_x\times\Omega_{y_2}\times\Omega_w\to \R$
   given by~\eqref{function_f}--\eqref{function_h} are smooth.
\end{rem}

\begin{rem}\label{trapping}
   Recall that the set $\Omega=\Omega_x\times\Omega_{y_2}\times\Omega_w$ is invariant (see Theorem~\ref{invariant}). In fact, it is even a trapping region, i.e. the solutions are contained in $\Int\Omega$ for $t>0,$ because the respective inequalities in the proof of the invariance of $\Omega$ become strict for $t>0$ (the vector field points inward everywhere on the boundary of $\Omega$).
For example, if $x_0=0$ we have $x_0'=\frac{\gamma_1}{\varepsilon_1}>0$ and for $x_0=1$ we deduce $x_0'=-\frac{y_2}{\varepsilon_1}<0$. 
	It means that the vector field is pointing to the left, so trajectories cannot leave the interior of the domain. Similar arguments apply to the variables $y_1$ and $z.$ 
	Namely, for $y_1=0,$ $y_1'=ky_2+\delta_1z>0$ for $y_1=0$ but for $y_1=\bar y_1,$ where $\bar y_1=r_0+\frac{k}{\delta_1}\theta_\gamma,$ since $\bar y_1>1,$ the following inequality is satisfied:
      \begin{align*}
      y_1'(\bar y_1)&=\frac{k\theta}{1+\theta x_0}\left(\gamma_1(1-x_0)-x_0\bar y_1^2\right)+\delta_1\bigl(z-\bar y_1\bigr)\\
      &< \frac{k\theta}{1+\theta x_0}\left(\gamma_1(1-x_0)-x_0\bar y_1\right)+\delta_1\bigl(z-\bar y_1\bigr).
      \end{align*}
      Then, by $z\leq r_0,$ $y_1'(\bar y_1)<0.$
	Moreover, for $z=0$ we have $z'=\delta_2r_0x_0>0,$ while $z=r_0$ we get $z'=\delta_2r_0(x_0-1)<0.$
	Combining these we deduce that the solution $(\bar x_0,\bar w )$ belongs to $\Int \Omega$.
\end{rem}

Now we prove two results on the global stability of the reduced form of~\eqref{nasz} (i.e. with or without dimers). First, we consider the model which describes changes of concentration of Hes1 dimers.
\begin{prop}\label{stab_g} 
   Let $g$ be the function defined by formula~\eqref{function_g}. For any fixed $x\in \Omega_x$ and $w\in \Omega_w$ there exists exactly one 
   non-negative steady state of equation  $\varepsilon_2 y_2'=g(x,y_2,w)$ which is globally asymptotically stable in $\Omega_{y_2}$ (uniformly with respect to $(x,w)=(x_0,\ldots,x_{n-1},y_1,z)$).
\end{prop}
\begin{proof}
   Note that we can rewrite the function $g$ in the following way
   \[
      g(x,y_2,w)=g(x_0,x_n,\dotsc,x_{n-1},y_2,y_1,z) = -\left(1+\theta\sum_{j=0}^{n-1} k_jx_j\right)y_2+
      \theta\left(\sum_{j=1}^{n-1}j\gamma_jx_j+n\gamma_n\biggl(1- \sum_{j=0}^{n-1}x_j\biggr)\right)  +y_1^2.
   \]      
Clearly, the equation $\varepsilon_2 y_2'=g(x,y_2,w)$ is linear with respect to $y_2$, as we consider fixed $x$ and $w=(y_1,z)$. Moreover, for any
   $x\in \Omega_x$ we have $x_0+x_1+\dotsc+x_{n-1}\le 1$, and $x_j\ge 0$, thus the equation $\varepsilon_2 y_2'=g(x,y_2,w)$
	has exactly one positive steady state which is globally asymptotically stable in $\Omega_{y_2}$, uniformly with respect to $(x,w)=(x_0,\ldots,x_{n-1},y_1,z)$ as the coefficient by $y_2$, that is $-1-\theta\sum_{j=0}^{n-1} k_jx_j$, is negative and separated from zero for all $x\in\Omega_x$.
	\qed
	\end{proof}

Next, we show stability of the system that describes the dynamics of number of free and occupied binding sites for the fixed concentration of dimers.
\begin{prop}\label{stab_f} 
Let $f$ be the function defined by formula~\eqref{function_f}. For any fixed $y_2\in \Omega_{y_2}$ there exists exactly one 
   non-negative steady state of equation $\varepsilon_1 x'=f(x,y_2)$ which is globally asymptotically stable in $\Omega_{x}\times\Omega_w$ (uniformly with respect to $y_2$).
\end{prop}
\begin{proof}
We note that the function $f$ is linear with respect to $x$, when $y_2$ is fixed. Thus, in order to determine the stability of the steady state it is enough to 
study the Jacobi matrix of the right hand side. Moreover, a unique steady state exists if and only if this matrix is non-singular. The existence of the non-negative steady state was already proved in Subsection~\ref{sec:1to2R}. We prove that all eigenvalues of the matrix of the right hand side of $\varepsilon_1 x'=f(x,y_2)$ are real and negative. To this end, we introduce additional variable and show that after the modification, the matrix becomes tridiagonal, and therefore, we use properties of tridiagonal matrices. Let us consider 
\[
   x_n = 1 - \bigl(x_0+x_1+x_2+\dots +x_{n-1}\bigr),
\]
and let $\hat x = (x,x_n) = (x_0,x_1,\dotsc,x_{n-1},x_n)$. The variable $\hat x$ fulfils the following linear ODE $\hat x' = A\hat x+b$, where  
\[
A=
\begin{bmatrix}
- k_0 y_2  & \widetilde\gamma_1         & 0  & 0                & \cdots                     & 0\\
k_0 y_2   & -k_1y_2-\widetilde\gamma_1 & \widetilde\gamma_2 &    0                  &  & & \\
0          & k_1y_2                     & -k_2y_2-\widetilde\gamma_2 & \widetilde\gamma_3 &  & \vdots\\
&  & \ddots &\ddots &\ddots& \\
\vdots     &                         &          &                 k_{n-2}\gamma_{n-1} &  -k_{n-1}y_2-\widetilde\gamma_{n-1}   & k_{n-1}y_2\\
0          &                          &   \cdots     &                    &  k_{n-1}y_2   & -\widetilde\gamma_n
\end{bmatrix},
\]
$\widetilde\gamma_1=j\gamma_j$, and the form of the vector $b$ is not important here. 
Now, the matrix $A$ is tridiagonal. We show that the eigenvalues of $A$ are real and the largest of them is equal to $0$. Because the studied system is linear (with respect to $y_2$) and $x_0+x_1+\dotsb+x_{n-1}+x_n=1$ is an invariant subspace for the system $\hat x' = A\hat x+b$ we deduce that the steady state of $\varepsilon_1 x'=f(x,y_2)$ is stable.

We observe that the tridiagonal matrix $A$ is similar to the following symmetric matrix
\[
P =
\begin{bmatrix}
\alpha_1  & \beta_1         & 0  & 0                & \cdots                     & 0\\
\beta_1  & \alpha_2 &   \beta_2                 &0  & \cdots & 0\\
0          & \beta_2                 & \alpha_3 & \beta_3 & \cdots & 0\\
\vdots&  \vdots & \ddots &\ddots &\ddots& \vdots\\
0     &   0                      &   \cdots       &                 \beta_{n-1}& \alpha_{n}  & \beta_{n}\\
0          &     0                     &   \cdots     &                 0   & \beta_{n} & \alpha_{n+1}
\end{bmatrix},
\]
where $\alpha_j=-k_{j-1}y_2-\widetilde\gamma_{j-1}$ are terms on the diagonal for $1\leq j\leq n+1$ and $\beta_j=\sqrt{k_{j-1}\widetilde\gamma_j y_2}$ for $1\leq j\leq n$ are terms on the super- and sub-diagonal. 
To shorten the notation, we set $k_n=0$ and $\widetilde\gamma_0=0$.
This matrix $P$ has a form $D_n^{-1}AD_n,$ where $D_n$ is a diagonal matrix with diagonal elements 
\[ 
\delta_1=1,\qquad \delta_j^2=\frac{k_0\ldots k_{j-2}}{\widetilde\gamma_1\ldots \widetilde\gamma_{j-1}}\,y_2^{j-1},\quad
j=2,\ldots n.
\]
The characteristic polynomial $\Delta_n(\lambda)= \det(P-\lambda I_n)$ can be computed by the following recurrence relations
\[
\Delta_0(\lambda)=1,\quad\Delta_1(\lambda)=\alpha_1-\lambda,\quad\Delta_j(\lambda)=(\alpha_j-\lambda)\Delta_{j-1}(\lambda)-\beta_{j-1}^2\Delta_{j-2}(\lambda)\quad \text{for}\quad j\geq 2.
\]
Since similar matrices have the same eigenvalues, it is enough to examine that the eigenvalues of $P$ are real and the largest of it is equal to $0$. Observe that all eigenvalues of $P$ are real and simple, because the matrix $P$ is irreducible, tridiagonal and symmetric
(see~\cite[Prop.~10.1.2]{Serre}).
Moreover, the general theory of tridiagonal matrices implies also (see~\cite[Thm~5.9 (The Sturm sequence property)]{Suli}) that the number of eigenvalues grater than some real number $a$ is equal to the number agreements of sign between consecutive members of the sign sequence $\left\{\Delta_0(a),\Delta_1(a),\ldots \Delta_n(a) \right\}$. It is clear that $\lambda=0$ is an eigenvalues of the matrix $A$ and thus it is an eigenvalue of the matrix $P$. 
Now, we prove that consecutive members of the sequence $\left\{\Delta_0(0),\Delta_1(0),\ldots \Delta_n(0) \right\}$ have different signs. To this end, we prove by mathematical induction on $n$ that 
\[
   \Delta_0(0)=1,\qquad \Delta_n(0)=(-1)^{n} k_0\ldots k_{n-1}y_2^n \quad {\text{for}}\quad n\geq 1.
\]
If $n=1$ this statement is obviously true, as $\Delta_1(0)=-k_0y_2.$ Now, we show that if the formula is correct for $\Delta_n(0)$ it is also correct for $\Delta_{n+1}(0)$. A direct calculation shows that $\Delta_2(0) = k_0k_1y_2^2$,  and
\[\Delta_{n+1}=\alpha_{n+1}\Delta_n-\beta_n^2\Delta_{n-1}
=(-k_ny_2-\tilde\gamma_n)(-1)^nk_0\ldots k_{n-1}y_2^n-k_{n-1}\tilde\gamma_ny_2(-1)^{n-1}k_0\ldots k_{n-2}y_2^{n-1}
= (-1)^{n+1} k_0\ldots k_{n}y_2^{n+1}.
\]
This equality ensures that successive terms of the sequence $\bigl(\Delta_{j}(0)\bigr)_{j=0,..,n}$ change their sign. 
Thus, the matrix $A$ has no eigenvalue grater than 0, so the system $(x_0,x_1,\dots,x_n)$ is stable (for fixed $y_2$ the system is linear).
In addition, as the system is linear the stability is uniform. For this purpose, it is enough to take a diminished set $\hat{\Omega},$ where $y_2$ is separated from zero (see Fig.~\ref{fig.izo}).
\qed 
\end{proof}

As mentioned before, there are two different ways to reduce system~\eqref{ogolny:epsilony_red} (or equivalently~\eqref{nasz}) depending on the time scales,
that is whether $\varepsilon_2\ll \varepsilon_1 \ll 1$ (the left part of Fig.~\ref{fig.diagram}) or $\varepsilon_1\ll \varepsilon_2 \ll 1$ (the right part of Fig.~\ref{fig.diagram}).
For better readability, we present a diagram (Fig.~\ref{fig.diagram2}) for our system~\eqref{nasz}.
{\small{\begin{figure}[h]
   \centerline{\includegraphics[width=0.8\textwidth]{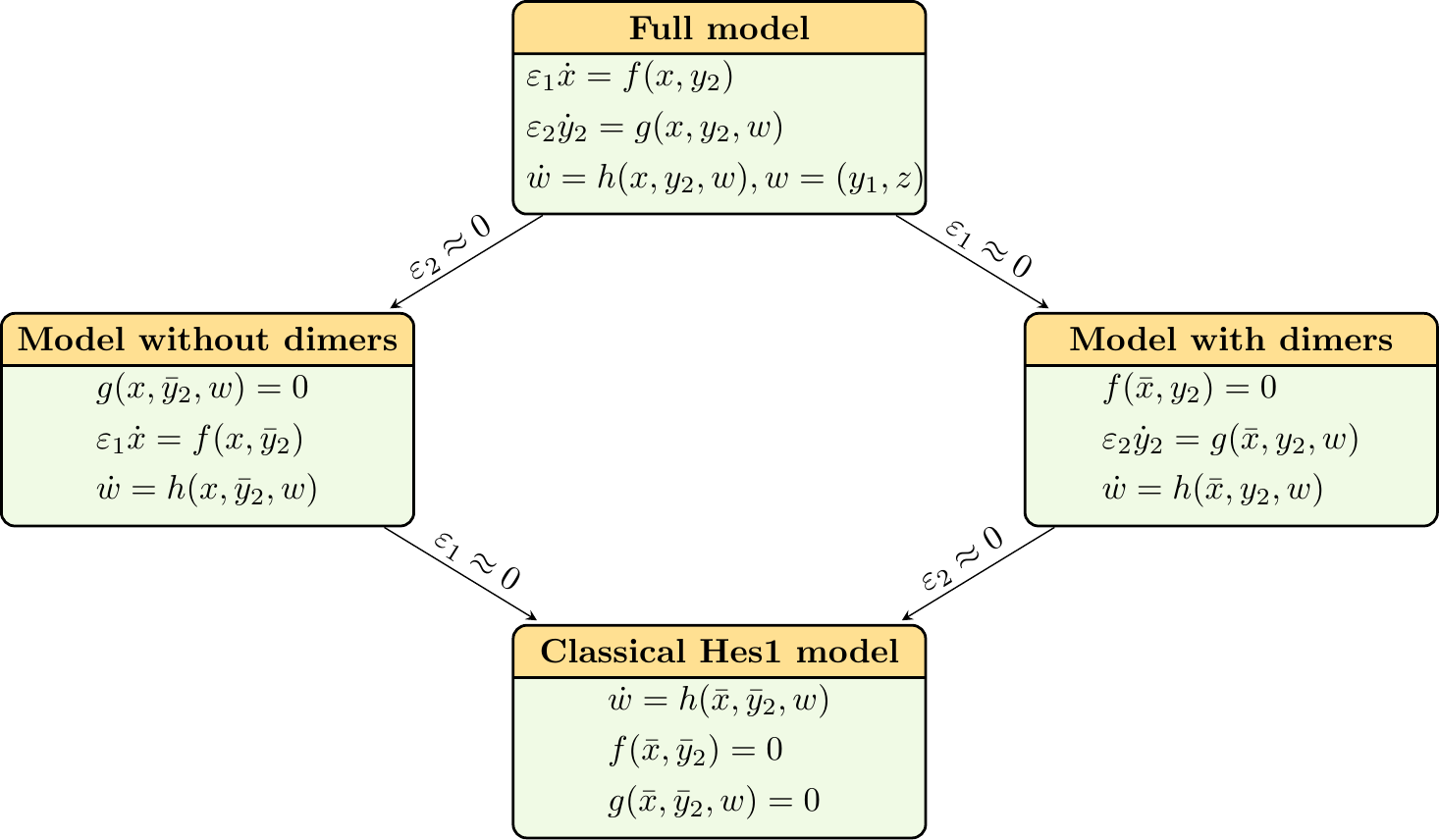}}
   \caption{Diagram of models created from the system~\eqref{nasz} and connection between them.\label{fig.diagram2}}
\end{figure}}}

\subsection{Application of The Tikhonov theorem}
Now we assume that the dimer creation is much faster process than other processes, so we assume that $\varepsilon_2$ is very small and we formulate the Tikhonov type theorem for this case. 
If we set $\varepsilon_2=0$, then system~\eqref{nasz}--\eqref{nasz_pocz} reduces to system~\eqref{uproszcz:n+2}, which has the form
\begin{equation}
\begin{split}\label{up}
\varepsilon_1 \bar x'&=f(\bar x,\varphi(\bar x, \bar w)),\\
\bar w'&=h(\bar x,\varphi(\bar x, \bar w),\bar w),
\end{split}
\qquad\qquad
\begin{split}
 \bar x(0)&=\mathring{x}\\
\bar w(0)&=\mathring{w},
\end{split}
\end{equation}
where 
$\varphi(x,w)=y_2$ is the solution of the equation $g(x,y_2,w)=0,$ the function $\varphi$ is given by \eqref{y2stacj} and $f$, $g$, $h$ are defined by \eqref{function_f}, \eqref{function_g}, \eqref{function_h}, respectively. 

\begin{thm}\label{full_2L}
   Assume that the functions $f\colon\Omega_x\times\Omega_{y_2}\to \R$ and $g,h\colon\Omega_x\times\Omega_{y_2}\times\Omega_w\to \R$ 
   are defined by~\eqref{function_f}--\eqref{function_h}. Then there exists $\varepsilon_0 >0,$ such that
   for any $\varepsilon_2\in(0,\varepsilon_0]$ there exists a unique solution  $(x_{\varepsilon_2}(t),y_{2,\varepsilon_2}(t),w_{\varepsilon_2}(t))$ to system~\eqref{nasz}--\eqref{nasz_pocz} on $[0,T]$ and the following conditions hold:
   \begin{align*}
      \lim_{\varepsilon_2\to 0} x_{\varepsilon_2}(t)&=  \bar x(t),\quad t\in [0,T]\\
      \lim_{\varepsilon_2\to 0} y_{2,\varepsilon_2}(t)&= \varphi(\bar x(t),\bar w(t)),\quad t\in (0,T]\\
      \lim_{\varepsilon_2\to 0} w_{\varepsilon_2}(t)&= \bar w(t),\quad t\in [0,T],
   \end{align*}
   where $(\bar x(t),\bar w(t))$ is the solution to system~\eqref{up} and constant $T$ 
   does not depend on $\varepsilon_2.$
\end{thm}

\begin{proof}
It is enough to check the assumptions (A1)--(A5) of Theorem~\ref{thm:Tikhonov}. The assumptions (A1), and (A2) are satisfied in consequence of the form of functions $f,$ $g$, and $h$. Assumption (A3) also holds as a consequence of Proposition~\ref{stab_g}. 
The existence of unique solution to system~\eqref{up} is obvious due to the form of the right-had side of the system.
Note also that by Remark~\ref{trapping} the set $\Omega_x\times\Omega_{y_2}\times\Omega_w$ is a trapping region for system~\eqref{nasz}, and the set $U=\Omega_x\times\Omega_w$ is a trapping region for system~\eqref{up}.
In the consequence the solution $(\bar x,\bar w)\in \Int U$, thus assumption~(A4) holds.
Finally, Assumption (A5) is satisfied due to Proposition~\ref{stab_g}, and applying the Tikhonov theorem completes the proof.
\qed
\end{proof}

Now, we assume that the creation and dissociation of Hes1 dimers are much faster than translation and transcription processes, that is $\varepsilon_1\ll 1$. For $\varepsilon_1=0$ system~\eqref{up} reduces to the classical Hes1 model~\eqref{zreduk}, which we rewrite as
\begin{equation}\label{classical}
 w'=h(\psi(w),\varphi(\psi(w),w),w),\qquad w(0)=\mathring{w},
\end{equation}
where $\psi(w)=x$ is the solution of the equation $f(x,\varphi(x,w))=0,$ and the function $\psi$ is given by~\eqref{bart:wzor_psi_ogolny}. Now, we state theorem saying that for $\varepsilon_1$ small enough, the solutions to system~\eqref{classical} approximate solutions to system~\eqref{up} well.
\begin{thm}\label{red_2L}
Assume that the functions $f\colon\Omega_x\times\Omega_{y_2}\to \R$ and $h\colon\Omega_x\times\Omega_{y_2}\times\Omega_w\to \R$ 
are defined by~\eqref{function_f} and~\eqref{function_h}, respectively. Then there exists $\varepsilon_0 >0,$ such that
for any $\varepsilon_1\in(0,\varepsilon_0]$ there exists a unique solution  
$(x_{\varepsilon_1}(t),w_{\varepsilon_1}(t))$ to system~\eqref{up} on $[0,T]$ and the following conditions hold:
\begin{align*}
\lim_{\varepsilon_1\to 0} x_{\varepsilon_1}(t)&=  \psi(\bar w(t)),\quad t\in (0,T]\\
\lim_{\varepsilon_1\to 0} w_{\varepsilon_1}(t)&=  \bar w(t),\quad t\in [0,T]
\end{align*}
where $\bar w(t)$ is the solution to system~\eqref{classical} and constant $T$ is independent of $\varepsilon_1.$
\end{thm}
\begin{proof}
It is easy to see that smoothness assumptions (A1) and (A2) of the Tikhonov theorem are satisfied. The stability properties, (A3) and (A5), as well as assumption (A4) hold, due to analogous argument as in the proof of Theorem~\ref{full_2L}, when we use Proposition~\eqref{stab_f} instead of Proposition~\ref{stab_g}. Thus, Theorem~\ref{thm:Tikhonov} yields the assertion of Theorem~\ref{red_2L}. 
\qed
\end{proof}

Assuming that the creation and dissociation of Hes1 dimers are much faster than translation and transcription processes, that is 
$\varepsilon_1$ is very small we reduce system~\eqref{nasz}--\eqref{nasz_pocz} to system~\eqref{uposzcz:n}, which reads
\begin{equation}
\begin{split}\label{up_right}
\varepsilon_2 \bar y_2'&=g(\varphi(\bar y_2),\bar y_2,\bar w),\\
\bar w'&=h(\varphi(\bar y_2),\bar y_2,\bar w),
\end{split}
\qquad\qquad
\begin{split}
 \bar y_2(0)&=\mathring{y_2}\\
\bar w(0)&=\mathring{w},
\end{split}
\end{equation}
where $\varphi(y_2)=x$ is the solution of the equation $f(x,y_2)=0,$ and the function $\varphi$ is given by~\eqref{y2stacj}.

\begin{thm}\label{full_2R}
Assume that the functions $f\colon\Omega_x\times\Omega_{y_2}\to \R$ and $g,h\colon\Omega_x\times\Omega_{y_2}\times\Omega_w\to \R$ 
are defined by~\eqref{function_f}--\eqref{function_h}. 
Then there exists $\varepsilon_0 >0,$ such that
for any $\varepsilon_1\in(0,\varepsilon_0]$ there exists a unique solution  $(x_{\varepsilon_1}(t),y_{2,\varepsilon_1}(t),w_{\varepsilon_1}(t))$ to system~\eqref{nasz}--\eqref{nasz_pocz} on $[0,T]$ and the following conditions hold:
\begin{align*}
\lim_{\varepsilon_1\to 0} x_{\varepsilon_1}(t)&=  \varphi(\bar w(t)),\quad t\in (0,T]\\
\lim_{\varepsilon_1\to 0} y_{2,\varepsilon_1}(t)&= \bar y_2(t),\quad t\in [0,T]\\
\lim_{\varepsilon_1\to 0} w_{\varepsilon_1}(t)&= \bar w(t),\quad t\in [0,T],
\end{align*}
where $(\bar y_2(t),\bar w(t))$ is the solution to system~\eqref{up_right} and constant $T$ is independent of $\varepsilon_1.$
\end{thm}

In the next step we put $\varepsilon_2\approx0$ and we reduce system~\eqref{up_right} to
the classical form
\begin{equation}\label{class}
\bar w'=h(\varphi(\psi(\bar w)),\psi(\bar w),\bar w),\qquad \bar w(0)=\mathring{w},
\end{equation}
where $\psi(w)=y_2$ is the solution of the equation $g(\varphi(\bar y_2),\bar y_2,\bar w)=0.$
\begin{thm}\label{red_2R}
Assume that the functions $g,h\colon\Omega_x\times\Omega_{y_2}\times\Omega_w\to \R$
are defined by~\eqref{function_g}--\eqref{function_h}. Then there exists $\varepsilon_0 >0,$ such that
for any $\varepsilon_2\in(0,\varepsilon_0]$ there exists a unique solution  
$(x_{\varepsilon_2}(t),w_{\varepsilon_2}(t))$ to system~\eqref{up} on $[0,T]$ and the following conditions hold:
\begin{align*}
\lim_{\varepsilon_2\to 0} x_{\varepsilon_2}(t)&=  \psi(\bar w(t)),\quad t\in (0,T]\\
\lim_{\varepsilon_2\to 0} w_{\varepsilon_2}(t)&=  \bar w(t),\quad t\in [0,T]
\end{align*}
where $\bar w(t)$ is the solution to system~\eqref{class} and constant $T$ does not depend on $\varepsilon_2$.
\end{thm}
The proofs of the Theorems~\ref{full_2R} and~\ref{red_2R} are analogously as Theorems~\ref{full_2L} and~\ref{red_2L}.
It is worth pointing out that the crucial assumption of these theorems is the Assumption (A3) of the Tikhonov theorem, i.e.
the steady state of the respective systems is asymptotically stable independently of the parameters value, which is guaranteed by Propositions~\ref{stab_f} and~\ref{stab_g}.

\section{Comparison of stability of the positive steady state}\label{sec:stab}
In this section we examine stability of the steady states for systems derived in Section~\ref{sec:fulsys}.
The first two subsections deal with the case of one binding site (i.e. $n = 1$) because with the increase of $n$ the complexity of the equations makes their analysis difficult.

\subsection{The full system --- the case $n=1$}\label{sec:fulln1}
Note that due to the scaling we have $k_0=1$. In~\eqref{ogolny:epsilony_red}, we wrote explicitly $k_0$ due to the symmetry of the notation. However, here we rewrite system~\eqref{ogolny:epsilony_red} for $n=1$ and $k_j$ for $j>0$
   are not present. Thus, we use the equality $k_0=1$ replacing 
   $k_0$ by~$1$, and the system reads
   \begin{equation}\label{sys:n=1}
   \begin{split}
   \varepsilon_1 x_0'&=\gamma_1(1-x_0)- y_2x_0,\\
   y_1'&=k(y_2-y_1^2)+\delta_1\bigl(z-y_1\bigr),\\
   \varepsilon_2 y_2'&=\theta(\gamma_1(1 -x_0)-x_0y_2)-y_2+y_1^2,\\	
   z'&=\delta_2\bigl(r_0x_0- z\bigr).
   \end{split}
   \end{equation}
 
    \begin{thm}\label{prop:full:n=1:stab}
       The positive steady state $(\bar x_0,1,1,1)$ of system~\eqref{sys:n=1} is locally asymptotically stable.
    \end{thm}  

\begin{proof}   
In order to calculate stability of the positive steady state $(\bar x_0,1,1,1)$ of system~\eqref{sys:n=1}, where $\bar x_0=1/r_0,$ $r_0=1+1/\gamma_1$, we rescale time to eliminate $\varepsilon_2$ from the left-hand side and then the Jacobi matrix reads
   \[
   \begin{bmatrix}
   -(\gamma_1+1)\varepsilon-\lambda  & 0 & -\bar x_0\varepsilon& 0\\
   0 & -(2k+\delta_1)\varepsilon_2-\lambda & k\varepsilon_2 & \varepsilon_2 \delta_1 \\
   -\theta(\gamma_1+1) & 2 & -(\theta \bar x_0+1)-\lambda & 0\\
   \varepsilon_2 \delta_2 r_0 & 0 & 0&  -\varepsilon_2 \delta_2-\lambda
   \end{bmatrix},
   \]
   where $\varepsilon=\frac{\varepsilon_2}{\varepsilon_1}$. 
   The characteristic polynomial in this case reads
   \begin{equation}\label{cha:n=1}
   W_1(\lambda) = \lambda^4 + a_1 \lambda^3 + a_2\lambda^2 + a_3\lambda + a_4,
   \end{equation}
   where 
   \begin{equation}\label{coeff}
   a_1 = \varepsilon_2\delta_2 +c_2, \quad 
   a_2 = \varepsilon_2\delta_2c_2 +c_1, \quad 
   a_3 = \varepsilon_2\delta_2c_1 +c_0,\quad 
   a_4 = \varepsilon_2\delta_2c_0 +2\varepsilon\varepsilon_2^2\delta_1\delta_2
   \end{equation}
   (we used the fact that $r_0\bar x_0=1 $),   
   and 
   \[
   \begin{split}
   c_0 &= \varepsilon\varepsilon_2 \delta_1 \bigl(\gamma_1+1\bigr)>0,\\
   c_1 &= \varepsilon\bigl(\gamma_1+1\bigr)\bigl(\varepsilon_2(2k+\delta_1)+1\bigr) 
   +\varepsilon_2\bigl(2k\theta\bar x_0+\delta_1(1+\theta \bar x_0)\bigr)>0, \\
   c_2 &= \varepsilon \bigl(\gamma_1+1\bigr) + \varepsilon_2\bigr(2k+\delta_1\bigl) + \theta \bar x_0 +1 >0.        
   \end{split}
   \]
   We check that all eigenvalues of the characteristic polynomial $W_1(\lambda)$ are negative.
   We observe that all coefficients $a_i$ of~\eqref{cha:n=1} are positive.
   Then according to the Hurwitz theorem, it is enough to show that
   \[
   a_1a_2a_3>a_3^2+a_1^2a_4,
   \]
   which is equivalent to $a_3(a_1a_2-a_3)>a_1^2a_4.$
   Note that the necessary condition $a_1a_2-a_3>0$ occurs
   \[
   \begin{split}
   a_1a_2-a_3 &=\left(\varepsilon_2\delta_2+c_2\right)\varepsilon_2\delta_2c_2+c_1c_2-c_0\\
   &=\varepsilon_2\delta_2c_2 \bigl(\varepsilon_2\delta_2+c_2\bigr) + c_3\bigl(c_0+c_4\bigr)+c_4>0,
   \end{split}
   \]
   where
   \[
   c_3 =\varepsilon(\gamma_1+1)+\varepsilon_2(2k+\delta_1)+\theta\bar x_0, \quad 
   c_4 = \varepsilon(\gamma_1+1)(2\varepsilon_2 k+1) + \varepsilon_2\Bigl(2k\theta\bar x_0+\delta_1(1+\theta \bar x_0)\Bigr).
   \]
   Note also that 
   \[
   a_4 = \varepsilon\varepsilon_2^2\delta_1\delta_2\bigl(\gamma_1+3\bigr)>0.
   \] 
Using the formulas for the coefficients $a_i,$ i.e.~\eqref{coeff}, after tedious but direct calculations, we obtain $a_3\bigl(a_1a_2-a_3\bigr)-a_1^2a_4>0$. We include these computations in the Appendix.
\qed
\end{proof}
   
\subsection{Model without dimers --- the case $n=1$}

For $n=1$ formula \eqref{y2stacj} takes the form
\[
y_2=\frac{y_1^2+\theta\gamma_1(1-x_0)}{1+\theta x_0}
\]
and using the equality $x_0+x_1=1$ we write system~\eqref{uposzcz:n} for $n=1$ as
\begin{equation}\label{uproszcz:n=1}
\begin{split}
\varepsilon_1 x_0'&=\frac{1}{1+\theta x_0}\left(\gamma_1(1-x_0)-x_0y_1^2\right),\\
y_1'&=\frac{k\theta}{1+\theta x_0}\left(\gamma_1(1-x_0)-x_0y_1^2\right)+\delta_1\bigl(z-y_1\bigr),\\
z'&=\delta_2\bigl(r_0x_0- z\bigr).
\end{split}
\end{equation}
We observe that the steady state is 
\[
\left(\frac{\gamma_1}{1+\gamma_1},1,1\right).
\]
\begin{thm}
   There exists a~unique positive steady state of system~\eqref{uproszcz:n=1}, which is asymptotically stable independently of the parameters value.
\end{thm}
\begin{proof}
   The characteristic matrix for the steady state $\left(\frac{\gamma_1}{1+\gamma_1},1,1\right)$ of~\eqref{uproszcz:n=1} reads
   \[
   \begin{bmatrix}
   -\frac{\eta (1+\gamma_1)}{\varepsilon_1} & - \frac{2\eta\gamma_1}{\varepsilon_1(1+\gamma_1)} & 0 \\
   -k\eta\theta(1+\gamma_1) & -\frac{2k\eta\theta\gamma_1}{1+\gamma_1} -\delta_1 & \delta_1 \\
   \delta_2 \frac{1+\gamma_1}{\gamma_1} & 0 & -\delta_2
   \end{bmatrix},
   \quad 
   \eta = \frac{1+\gamma_1}{1+\gamma_1(1+\theta)}
   \]   
   Then the characteristic function takes the form
   \[
      W(\lambda)=\lambda^3+a_1\lambda^2+a_2\lambda+a_3,
   \]
   where
   \begin{equation}\label{wsp_n=1}
   \begin{split}
   a_1&=\delta_1+\delta_2+\eta\frac{1+\gamma_1}{\varepsilon_1}+2k\eta\theta\frac{\gamma_1}{1+\gamma_1},\\
   a_2&=\eta\bigl(\delta_1+\delta_2\bigr)\frac{1+\gamma_1}{\varepsilon_1}+2k \eta\theta \delta_2\frac{\gamma_1}{1+\gamma_1}+\delta_1\delta_2,
   \\
   a_3&=\frac{\eta\delta_1\delta_2}{\varepsilon_1}(3+\gamma_1).
   \end{split}
   \end{equation}   
   For the polynomial $W(\lambda)$ of degree $3$ the Routh-Hurwitz criterion simplify to $a_1,a_3>0$ and $a_1a_2>a_3.$
   Note that $a_1,a_3>0$ and the condition $a_1a_2>a_3$ is equivalent to 
   \[
   \left(\delta_1+\delta_2+\eta\frac{1+\gamma_1}{\varepsilon_1}+2k\eta\theta\frac{\gamma_1}{1+\gamma_1}\right)
   \left(\eta\bigl(\delta_1+\delta_2\bigr)\frac{1+\gamma_1}{\varepsilon_1}+2k \eta\theta \delta_2\frac{\gamma_1}{1+\gamma_1}+\delta_1\delta_2\right)
   >\frac{\eta\delta_1\delta_2}{\varepsilon_1}(3+\gamma_1).
   \]
Moreover, the above inequality is always fulfilled. Namely, multiplication of $\delta_1 + \delta_2$ from the first bracket by the first term of the second bracket gives 
$\frac{\eta}{\varepsilon_1}(\delta_1+\delta_2)^2(1+\gamma_1)$. Adding to this expression the second term of the first bracket multiplied by the last term of the second bracket gives 
   \[
      \frac{\eta}{\varepsilon_1}(\delta_1+\delta_2)^2(1+\gamma_1) + \frac{\eta\delta_1\delta_2}{\varepsilon_1}(1+\gamma_1) = 
      \frac{3\eta\delta_1\delta_2}{\varepsilon_1}(1+\gamma_1) + \frac{\eta}{\varepsilon_1}(\delta_1^2+\delta_2^2)(1+\gamma_1) 
      > \frac{\eta\delta_1\delta_2}{\varepsilon_1}(3+\gamma_1).     
   \]
   This completes the proof.
\qed
\end{proof}

\subsection{Model with dimers}

Now, we formulate theorem considering stability of the steady state $(1,1,1)$ of the system~\eqref{uposzcz:n}.
\begin{thm}\label{cond_stab}
If the function $\psi$ given by~\eqref{bart:wzor_psi_ogolny} satisfies the following inequality
      \begin{equation}\label{psi}
      -\psi'(1) < \frac{\Bigl(\varepsilon_2\bigl(2k+\delta_1\bigr)+1\Bigr)
         \Bigl(\varepsilon_2\delta_2\bigl(2k+\delta_1+\delta_2\bigr)+\delta_1+\delta_2\Bigr)}
      {2\varepsilon_2 r_0 \delta_1\delta_2},
      \end{equation}
   then the steady state $(1,1,1)$ of system~\eqref{uposzcz:n} is locally asymptotically stable. 
\end{thm}
\begin{proof}
   Linearising the system~\eqref{uposzcz:n} around the steady state $(1,1,1)$ we obtain 
   the following matrix 
      \[
      \begin{bmatrix} 
      -A~& k & \delta_1 \\
      2/\varepsilon_2 & -1/\varepsilon_2 & 0 \\
      0 & r_0\delta_2\psi'(1) & -\delta_2
      \end{bmatrix},
      \quad A~= 2k +\delta_1.
      \]	
   The characteristic polynomial reads 
        \[
      W(\lambda) = \lambda^3 +\biggl(A+\delta_2+\frac{1}{\varepsilon_2}\biggr) \lambda^2 
      +\Biggl(\frac{\delta_1}{\varepsilon_2}+\delta_2\biggl(A+\frac{1}{\varepsilon_2}\biggr)\Biggr) \lambda 
      + \frac{\delta_1\delta_2}{\varepsilon_2}\Bigl(1-2r_0\psi'(1)\Bigr).
      \] 
   Due to positivity of the coefficient and negativity of $\psi'(1)$ we can immediately 
   conclude that all coefficients of polynomial $W$ are positive. 
   Thus, due to the Routh-Hurwitz criterion, 
   the steady state is locally stable if the inequality 
   \[ 
   \biggl(A+\delta_2+\frac{1}{\varepsilon_2}\biggr)\Biggl(\frac{\delta_1}{\varepsilon_2}+\delta_2\biggl(A+\frac{1}{\varepsilon_2}\biggr)\Biggr)> \frac{\delta_1\delta_2}{\varepsilon_2}\left(1-2r_0\psi'(1)\right)
   \]
   which is equivalent to
    \[
   \Bigl(A\varepsilon_2+1\Bigr)\Bigl(\delta_1+\delta_2(A\varepsilon_2+1)+\varepsilon_2\delta_2^2\Bigr)
   > 
   -2r_0\delta_1\delta_2\varepsilon_2 \psi'(1).
   \]
This completes the proof.
\qed
\end{proof}
\begin{rem}
If the strict inequality reverse to~\eqref{psi} holds, that is 
\begin{equation}\label{psi:nstab}
      -\psi'(1) > \frac{\Bigl(\varepsilon_2\bigl(2k+\delta_1\bigr)+1\Bigr)
         \Bigl(\varepsilon_2\delta_2\bigl(2k+\delta_1+\delta_2\bigr)+\delta_1+\delta_2\Bigr)}
      {2\varepsilon_2 r_0 \delta_1\delta_2},
\end{equation}
then the steady state $(1,1,1)$ of system~\eqref{uposzcz:n} is unstable.
\end{rem}

Recall that due to the chosen scaling equality $\psi(1)=\frac{1}{r_0}$ holds, where 
$r_0$ is defined by~\eqref{rown_r0}, i.e.
\begin{equation*}
   r_0=1+\sum_{j=1}^{n}\frac{1}{j!}\frac{\tilde k_0\ldots \tilde k_{j-1}}{\tilde \gamma_1\ldots\tilde \gamma_j}.
   \end{equation*}
The stability condition presented in Theorem~\ref{cond_stab} is 
valid for an arbitrary $C^1$ class function $\psi$. However, due to the origin of system~\eqref{uposzcz:n}, 
the function $\psi$ has a specific form given by~\eqref{bart:wzor_psi_ogolny}. 
Therefore, both sides of inequality~\eqref{psi} depend on the number of binding sites $n$. 
In the following, we prove that if the number of binding sites is small (not grater than 4) 
then the steady state is always stable, regardless of the value of other parameters. 
The destabilisation of the steady state is possible if there is at least 5 binding sites. 

\begin{thm}\label{thm:stabn} 
   If the function $\psi$ is given by~\eqref{bart:wzor_psi_ogolny}, then
      \begin{enumerate}
      \item for $n\le 4$ the steady state of system~\eqref{uposzcz:n} is locally asymptotically stable,
      \item for $n\ge 5$ there exists set of parameters of model~\eqref{uposzcz:n} as well as the function
      $\psi$ such that  the steady state of system~\eqref{uposzcz:n} is unstable. More precisely, if $n\ge 5,$
			$\delta_1$, $\delta_2$ and $\varepsilon_2$ sufficiently close to 1, $k$ small enough, $\psi(y) = \dfrac{1}{1+(r_0-1)y^n}$ and 
			 $r_0>\dfrac{n}{n-4},$ then the steady state of system~\eqref{uposzcz:n} is unstable.
			\end{enumerate}
\end{thm}

\begin{lem}\label{lem:maxdg}
   Let $\psi(y) = \frac{1}{1+q(y)}$, where $q$ is a~polynomial of $n$-th degree with 
   all coefficient non-negative such that $q(0)=0$. Then 
   \[
   -\psi'(1) \le \frac{nq(1)}{(1+q(1))^2}.
   \]
\end{lem}
\begin{proof}
   Let us fix the value $q(1)$. We obtain that 
   \[
   \psi'(1)=-\frac{q'(1)}{(1+q(1))^2}.
   \]
   We need to find an upper bound of the numerator of $|\psi'(1)|$. 
   Let $q(y) = \displaystyle\sum_{j=1}^n a_j x^j$. Then this
   numerator reads $\displaystyle\sum_{j=1}^n j a_j$. Since all coefficients $a_j$ of $q$ are non-negative, we have
   \[
   q'(1)=\sum_{j=1}^n j a_j=nq(1)-\sum_{j=1}^{n-1} (n-j)a_j\le n q(1),
   \]
   which proves our assertion.
\qed
\end{proof}

\begin{prthm}
  Let us prove the first part of our theorem.
	By~\eqref{rown_r0} and~\eqref{bart:wzor_psi_ogolny} we get $\psi(1)=1/r_0.$
  Using notation of Lemma~\ref{lem:maxdg} we get $r_0=1+ q(1)$. Thus,
  due to Lemma~\ref{lem:maxdg} we obtain
   \[
      -\psi'(1)\le \frac{n(r_0-1)}{r_0^2}.
   \]
   Suppose that the steady state is not locally asymptotically stable. Then, the inequality reverse to~\eqref{psi} holds and we get 
    \[
      \frac{n(r_0-1)}{r_0^2}\ge -\psi'(1)\geq  \frac{\Bigl(\varepsilon_2\bigl(2k+\delta_1\bigr)+1\Bigr)
         \Bigl(\varepsilon_2\delta_2\bigl(2k+\delta_1+\delta_2\bigr)+\delta_1+\delta_2\Bigr)}
      {2\varepsilon_2 r_0 \delta_1\delta_2}.
      \]
    Thus,  
    \begin{equation}\label{nier:na:n}
    n \ge \frac{r_0}{r_0-1} \frac{\Bigl(\varepsilon_2\bigl(2k+\delta_1\bigr)+1\Bigr)
       \Bigl(\varepsilon_2\delta_2\bigl(2k+\delta_1+\delta_2\bigr)+\delta_1+\delta_2\Bigr)}
    {2\varepsilon_2 \delta_1\delta_2}.
    \end{equation}
    Observe, that as $k>0$, the second term of~\eqref{nier:na:n} can be estimated as
    \begin{equation}\label{est:for:second:term}
      \frac{\Bigl(\varepsilon_2\bigl(2k+\delta_1\bigr)+1\Bigr)
         \Bigl(\varepsilon_2\delta_2\bigl(2k+\delta_1+\delta_2\bigr)+\delta_1+\delta_2\Bigr)}
      {2\varepsilon_2 \delta_1\delta_2}\ge 
      \frac{\bigl(\delta_1+\delta_2\bigr)^2}{2\delta_1\delta_2} + \frac{1}{2}\Biggl(
      \biggl(\frac{1}{\varepsilon_2\delta_1} + \varepsilon_2\delta_1\biggr)+
      \biggl(\frac{1}{\varepsilon_2\delta_2} + \varepsilon_2\delta_2\biggr) \Biggr)\ge 4.
    \end{equation}
    Combining~\eqref{nier:na:n} with \eqref{est:for:second:term} we obtain
    \[
      n \ge \frac{4r_0}{r_0-1} > 4,
    \]
		which contradicts our assumption.
		
   To prove the second part of our theorem we observe that inequalities~\eqref{est:for:second:term} becomes equalities 
    for $\varepsilon_2=\delta_1=\delta_2=1$ and $k=0$. Thus, for $n\ge 5$, for $r_0$ sufficiently large, that is grater than
    \[
       r_0>\frac{n}{n-4}
    \] 
    and for sufficiently small $k$ inequality~\eqref{nier:na:n} holds. As inequality proved in Lemma~\ref{lem:maxdg} becomes equality for $q(y)=a y^n$, inequality \eqref{psi:nstab} holds for $n\ge 5,$ $k$ sufficiently small, and $\delta_1$, $\delta_2$ sufficiency close to $1$. This completes the proof.    
\qed
\end{prthm}

\section{Discussion and conclusions}\label{sec:concl}

The theorems stated in the previous sections show, that the dynamics of the reduced model (a slow part of model) is similar to the model before reduction (that contains both parts: slow and fast) if the ,,fast'' subsystem is sufficiently fast (that is $\varepsilon_1$ or $\varepsilon_2$ is sufficiently small). On the other hand, we have also showed that the dynamics of such model can differ from the dynamics of the reduced system. In particular, in the classical Hes1 system~\eqref{zreduk} without time delay the positive steady state is always (globally) asymptotically stable. However, if we take into account Hes1 dimers as a separate population, that is system~\eqref{uposzcz:n}, we can see that the steady state can loose stability and oscillatory behaviour is possible (if the number of binding sites is grater than 4). Now, we numerically illustrate similarities and differences in behaviour of solutions to the considered models.

\subsection{Numerical simulations}
In order to show behaviour of four models considered in this paper
(full, without dimers, with dimers and classical),
basing on data from~\cite{monk03currbiol}, 
we choose the following parameters
 \begin{equation}\label{par_wsp}
      \delta_1 = 0.2242, \quad 
      \delta_2 = 0.2075, \quad 
      \theta =0.5 \quad \text{and} \quad 
      \varepsilon_1=1.
 \end{equation}  
Our models, before scaling, consist a vast number of parameter. Estimating those parameter may be a challenging problem that we are not going to address here. Our aim is only to illustrate possible models' behaviours. Thus, we choose the parameter values only for the non-dimensional version of our models ensuring that the ratio between $\delta_1$ and $\delta_2$ agrees with protein decay and mRNA decay rates considered in~\cite{monk03currbiol}. We illustrate the behaviour of considered models for two different situations.
In the first case, we consider the situation, when the steady states of all models are stable and oscillatory behaviour is not possible. 
We put 
 \begin{equation}\label{par1}
   k = 0.2, \quad   
   \varepsilon_2 = 1, \quad 
   k_0 = 1, \quad 
   k_1 = 1.5, \quad 
   k_2 = 1.5,\quad 
   \gamma_1 = 1,\quad 
   \gamma_2 = 0.5, \quad 
   \gamma_3 = 0.5.  
 \end{equation}
Here, we take $\gamma_i$ and $k_i$ to reflect the cooperative character of the binding of dimer to DNA promoter. Thus, binding ratio is larger  and the dissociation parameter is smaller if at least one binding site is occupied. The result of an exemplary simulation is presented in Fig.~\ref{fig:zbieznosc}. Although the difference in solutions to different models is visible, the qualitative behaviour of those solutions are the same. In particular, all solutions converge to the stationary state. 

\begin{figure}[bht]
   \centerline{\includegraphics{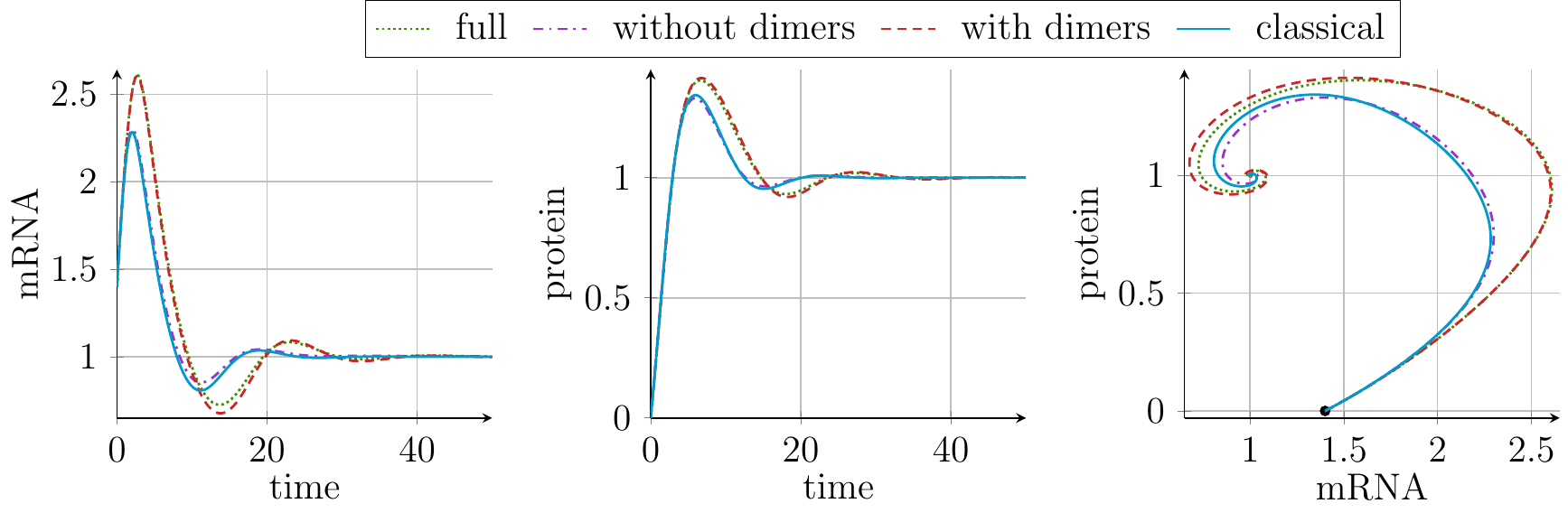}}
   \caption{Comparison of solutions to system considered in the paper for the case $n = 3$. Parameters as given in~\eqref{par_wsp} and~\eqref{par1}.}\label{fig:zbieznosc}
\end{figure}
 
On the other hand, we also want to illustrate possible oscillatory behaviour of system with dimers~\eqref{uposzcz:n} if the derivative of function $\psi$ (given by~\eqref{bart:wzor_psi_ogolny}) at the stationary state is large enough. In order to do that we need to take $n\ge 5$. We decided to take the smallest possible $n=5$, but then we need to take $k$ small enough and appropriate $k_i$ and $\gamma_i$ such that the inequality revers to~\eqref{psi} holds. We take
 \begin{equation}\label{par2}
    k = 0.01, \quad   
    \varepsilon_2 = 5, \quad 
    k_0 = 1, \quad 
    k_1 = 2, \quad 
    k_2 = 3,\quad 
    k_3 = 4, \quad 
    k_4 = 700, \quad 
    \gamma_1 = 2,\quad 
    \gamma_i = 1, \; i=2,\ldots,5.\quad 
 \end{equation}
 The results of the simulations are presented in Fig.~\ref{fig:oscylacje}. The interesting fact is that only solutions to~\eqref{uposzcz:n} exhibits oscillatory behaviour. Of course, if we take $\varepsilon_1$ sufficiently close to zero, the oscillatory behaviour appears also for
the full model~\eqref{ogolny:epsilony_red}. 

\begin{figure}[h!]
   \centerline{\includegraphics{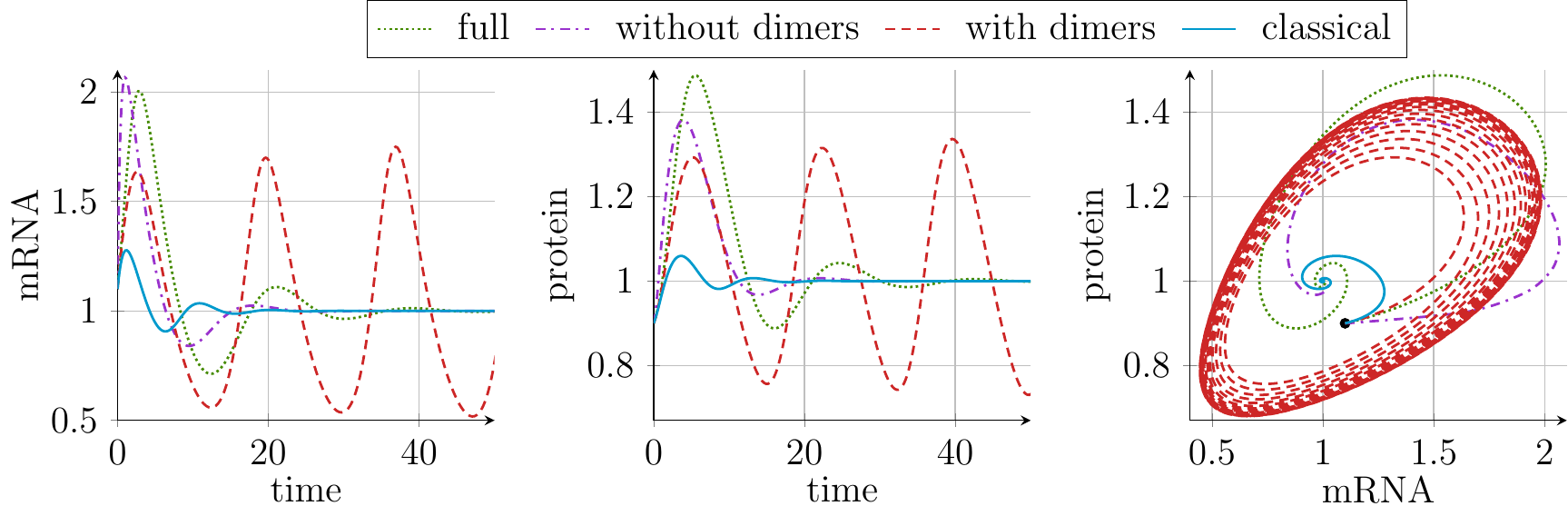}}
	\caption{Comparison of solutions to system considered in the paper for the case $n = 5$. Parameters as given in~\eqref{par_wsp} and~\eqref{par2}.}\label{fig:oscylacje}
\end{figure}

Note that we need to take $k_4$ very large in order to fulfil condition~\eqref{psi:nstab} for $n=5$ and $\delta_1$, $\delta_2$ as in~\eqref{par_wsp}. This causes that we need very small $\varepsilon_1$ to get an agreement ib behaviour of solutions to the model with dimers and to the full model. However, if we consider larger number of binding sites, we may take $k_i$ of order 1. For example, for $n=9$ of binding sites and $k_i=i+1$ we observe oscillatory behaviour of the solution to the model with dimers and dumping oscillation in the full model for $\varepsilon_1=1$ (compare the upper row of Fig.~\ref{fig:oscylacje2}) or oscillation to both models for $\varepsilon_1=0.05$ (compare the bottom row of Fig.~\ref{fig:oscylacje2}). 

\begin{figure}[bht]
   \centerline{\includegraphics{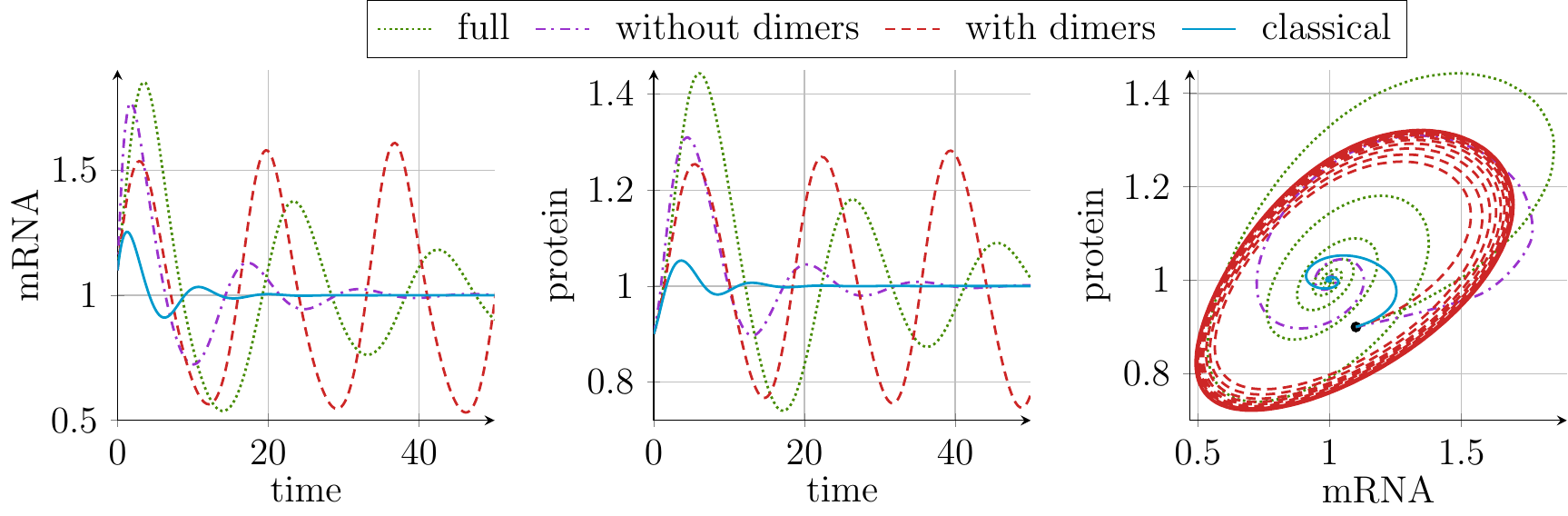}}
   \centerline{\includegraphics{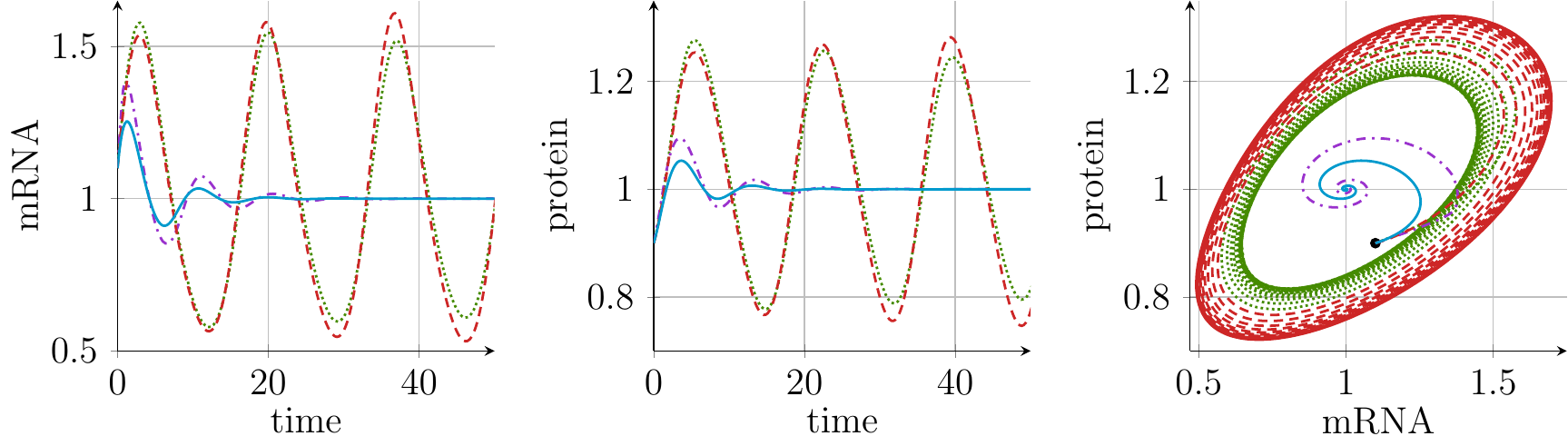}}
   \caption{Comparison of solutions to system considered in the paper for the case $n = 9$. Parameters as given in~\eqref{par_wsp} and~\eqref{par2}, except that $k_i=i+1,$ $\gamma_i=1$ and $\varepsilon_1=1$ (upper row) or $\varepsilon_1=0.05$ (lower row).}\label{fig:oscylacje2}
\end{figure}

\subsection{Conclusions}
In this paper we examine a gene expression model of the self-regulating Hes1 protein. We take into account the existence of multiple binding sites in the Hes1 promoter and the model of the transcriptional and translational processes. In our model, we include negative feedback and we take into account the formation of dimer complexes by proteins before blocking DNA. This model is more complex than the one studied earlier in~\cite{bernard06philtrans,mbab12nonrwa,monk03currbiol}. Using the Tikhonov theorem we show that if the process of binding to and dissolving from the binding sites and the process of dimer formation are much faster then the complex full model can be reduced to more simple ones. In particular we rigorously justify that formal quasi-stationary approach for simplifying models make sense in this case. We also examine the dynamics of the considered models. It turn out that the dynamics of model in which we take into account dimer formation and separate population of dimers, differs from the dynamics of classical Hes1 model. For some range of parameter the oscillatory behaviour is possible. This is not the case of classical Hes1 model proposed by Monk~\cite{monk03currbiol} unless time delay is taken into account. However, our findings show that the multiple binding sites are needed to get such behaviour. We show, that only one binding site, the steady state is stable in all considered models. On the other hand, for the reduced model with dimer formation~\eqref{uposzcz:n}, the minimal number of biding sites that would allow such behaviour (if neglecting time delay) is 5. 

Summarizing, we prove that the quasi-stationary approach is justified for this family of models, and we observe that the number of active transcription factor binding sites is essential for the dynamical behaviour of the Hes1 model.

\appendix
\section{The calculation needed in the proof of stability of the steady state of system~\eqref{sys:n=1}}
Here we include the calculations that proves the stability of the steady state of system~\eqref{sys:n=1}. The notation is as in Section~\ref{sec:fulln1}, in particular $a_i$ are defined by~\eqref{coeff}.
\[
\begin{split}
&a_3\bigl(a_1a_2-a_3\bigr)-a_1^2a_4 =\\
& 
\varepsilon_2 \na{\Bigl(}\varepsilon_2^2 \nb{\Bigl(}\zeta \delta_2^2 \varepsilon_2^2+\nc{\Bigl(}\varepsilon^2+
\nd{\Bigl(}\theta  \bar x_0+\varepsilon  \gamma_1\nd{\Bigr)} \nd{\Bigl(}\zeta+\varepsilon+1\nd{\Bigr)}+1\nc{\Bigr)} \delta_2 \varepsilon_2+
\varepsilon\zeta \nc{\Bigl(}\gamma_1+1\nc{\Bigr)} 
\nb{\Bigr)} \delta_1^3\\
&+\varepsilon_2 
\nb{\Bigl(}\delta_2^2 \nc{\Bigl(}2 k \nd{\Bigl(}3\zeta-1\nd{\Bigr)}+
\zeta \delta_2\nc{\Bigr)} \varepsilon_2^3+
2 \delta_2 \nc{\Bigl(}k\nd{\Bigl(}3 \varepsilon ^2+\nf{\Bigl(}\theta  \bar x_0+\varepsilon  \gamma_1\nf{\Bigr)}
\nf{\Bigl(}\zeta+3\varepsilon+1\nf{\Bigr)}+1\nd{\Bigr)}+\nd{\Bigl(}\varepsilon ^2+\nf{\Bigl(}\theta 
\bar x_0+\varepsilon  \gamma_1\nf{\Bigr)} \nf{\Bigl(}\zeta+\varepsilon+1\nf{\Bigr)}+1\nd{\Bigr)} \delta_2\nc{\Bigr)} \varepsilon_2^2\\
&+
\nc{\Bigl(}2 k \varepsilon  \nd{\Bigl(}\gamma_1+1\nd{\Bigr)} \nd{\Bigl(}\zeta-1\nd{\Bigr)}+\zeta 
\nd{\Bigl(}\theta ^2 \bar x_0^2+2 \theta  \nf{\Bigl(}\gamma_1 \varepsilon +\varepsilon +1\nf{\Bigr)}
\bar x_0+
(\varepsilon -1) \varepsilon +\varepsilon  \gamma_1 \nf{\Bigl(}\gamma_1
\varepsilon +2 \varepsilon +3\nf{\Bigr)}+1\nd{\Bigr)} 
\delta_2\nc{\Bigr)} \varepsilon
_2+\varepsilon  \nc{\Bigl(}\gamma_1+1\nc{\Bigr)} \zeta^2\nb{\Bigr)} \delta_1^2+\\
&+
\nb{\Bigl(}2 k \delta_2^2 \nc{\Bigl(}6k\zeta-4k+\delta_2(\zeta-1)\nc{\Bigr)} \varepsilon_2^4+
\delta_2
\nc{\Bigl(}4\nd{\Bigl(}3\varepsilon^2+(\theta\bar x_0+\varepsilon\gamma_1)(3\zeta+3\varepsilon-1)\nd{\Bigr)} k^2+\\
&+
2 \nd{\Bigl(}4 \theta ^2 \bar x_0^2+2 \theta 
\nf{\Bigl(}4 \gamma_1 \varepsilon +4 \varepsilon +3\nf{\Bigr)} \bar x_0+\varepsilon  (4
\varepsilon +3)+\varepsilon  \gamma_1 \nf{\Bigl(}4 \gamma_1 \varepsilon +8
\varepsilon +7\nf{\Bigr)}+2\nd{\Bigr)} \delta_2 k+\\
&+
\nd{\Bigl(}\varepsilon
^2+\nf{\Bigl(}\theta  \bar x_0+\varepsilon  \gamma_1\nf{\Bigr)} \nf{\Bigl(}\zeta+\varepsilon+1\nf{\Bigr)}+1\nd{\Bigr)} \delta
_2^2\nc{\Bigr)} \varepsilon_2^3+\nc{\Bigl(}4 \varepsilon  \nd{\Bigl(}\gamma
_1+1\nd{\Bigr)} \nd{\Bigl(}\zeta-1\nd{\Bigr)} k^2+\\
&+
2 \nd{\Bigl(}2 \varepsilon ^3 \gamma_1^3+\varepsilon ^2
\nf{\Bigl(}6 \varepsilon +6 \theta  \bar x_0+7\nf{\Bigr)} \gamma_1^2+\varepsilon 
\nf{\Bigl(}2 (\varepsilon +1) (3 \varepsilon +2)+\theta  \bar x_0 \nh{\Bigl(}12
\varepsilon +6 \theta  \bar x_0+11\nh{\Bigr)}\nf{\Bigr)} \gamma_1+\varepsilon ^2 (2\varepsilon +3)+\\
&+
\theta  \bar x_0 \nf{\Bigl(}\varepsilon  (6 \varepsilon +7)+2 \theta
\bar x_0 \nh{\Bigl(}3 \varepsilon +\theta  \bar x_0+2\nh{\Bigr)}+2\nf{\Bigr)}\nd{\Bigr)}
\delta_2 k+\\
&+
\nd{\Bigl(}\zeta\nd{\Bigr)} \nd{\Bigl(}\theta ^2 \bar x_0^2+2 \theta  \nf{\Bigl(}\gamma_1
\varepsilon +\varepsilon +1\nf{\Bigr)} \bar x_0+(\varepsilon -1) \varepsilon +\varepsilon
\gamma_1 \nf{\Bigl(}\gamma_1 \varepsilon +2 \varepsilon
+3\nf{\Bigr)}+1\nd{\Bigr)} \delta_2^2\nc{\Bigr)} \varepsilon_2^2+\\
&+
2 \varepsilon 
\nc{\Bigl(}\gamma_1 \delta_2\zeta^2+k \nd{\Bigl(}\gamma_1+1\nd{\Bigr)} \nd{\Bigl(}\zeta^2-1-\theta\bar x_0\nd{\Bigr)} \varepsilon_2+\varepsilon ^2
\nc{\Bigl(}\gamma_1+1\nc{\Bigr)}^2 \zeta\nb{\Bigr)} \delta_1+\\
&+
\delta_2 \nb{\Bigl(}\zeta +2 k \varepsilon_2\nb{\Bigr)}
\nb{\Bigl(}\gamma_1 \varepsilon +\varepsilon +2 k \nc{\Bigl(}\zeta-1\nc{\Bigr)} \varepsilon_2\nb{\Bigr)} \nb{\Bigl(}\gamma_1
\nc{\Bigl(}\nd{\Bigl(}2 k+\delta_2\nd{\Bigr)} \varepsilon_2+1\nc{\Bigr)} \varepsilon
+\varepsilon +
\varepsilon_2 \nc{\Bigl(}2 k \varepsilon +(\varepsilon +1) \delta_2+\nd{\Bigl(}2 k+\delta_2\nd{\Bigr)}\nd{\Bigl(}\theta  \bar x_0+\delta_2
\varepsilon_2\nd{\Bigr)}\nc{\Bigr)}\nb{\Bigr)}\na{\Bigr)},
\end{split}
\]
where 
\[
\zeta= \varepsilon  \gamma_1+\varepsilon+\theta  \bar x_0+1.
\]
Since the following estimation holds
\[
(\varepsilon -1) \varepsilon +\varepsilon  \gamma_1 \Bigl(\gamma_1 \varepsilon +2 \varepsilon +3\Bigr)+1
\ge (\varepsilon -1) \varepsilon +1 =
\varepsilon^2 -\varepsilon + 1 \ge 0
\]
for all $\varepsilon\ge 0,$ an easy calculation shows that
$a_3(a_1a_2-a_3)-a_1^2a_4>0.$

\section*{Acknowledgements}
This work was partially supported by grant no. 2015/19/B/ST1/01163 of National Science Centre, Poland (MB) and by The Warsaw Center of Mathematical and Computer Science program ``Guests and small scientific meetings'' (AB). 

\bigskip
%

\end{document}